\documentclass[11pt,reqno]{amsart}
\usepackage{fullpage}
\usepackage{appendix}
\usepackage{graphicx}
\usepackage{subcaption}
\usepackage{amsthm}
\usepackage{mathrsfs} 
\usepackage{amsfonts}
\usepackage{amssymb}
\usepackage{diagbox}
\usepackage{mathtools}
\usepackage{longtable} 
\usepackage{indentfirst}
\usepackage{listings,xcolor}
\usepackage[colorlinks=true,linkcolor=blue]{hyperref}
\numberwithin{equation}{section}
\usepackage{enumerate}
\usepackage{MnSymbol}
\theoremstyle{plain}
\newtheorem{theorem}{Theorem}[section]

\newtheorem{lemma}[theorem]{Lemma}
\newtheorem{corollary}[theorem]{Corollary}
\newtheorem{proposition}[theorem]{Proposition}

\theoremstyle{remark}
\newtheorem{remark}{Remark}[section]

\numberwithin{equation}{section}
\allowdisplaybreaks[2]

\newcommand{\bseq}{\mathbf{b}}
\newcommand{\aseq}{\mathbf{a}}
\newcommand{\cseq}{\mathbf{c}}
\newcommand{\dset}{\mathcal{D}}
\newcommand{\bd}{\mathbf{b},\mathcal{D}}
\newcommand{\betaD}{\beta,\mathcal{D}}
\newcommand{\dmax}{\delta_{\max}}
\newcommand{\dmin}{\delta_{\min}}

\newcommand{\bt}{\mathbb{T}}

\newcommand{\br}{\mathbb{R}}
\newcommand{\bn}{\mathbb{N}}
\newcommand{\bz}{\mathbb{Z}}
\newcommand{\bc}{\mathbb{C}}

\newcommand{\cR}{\mathcal{R}}

\newcommand{\cN}{\mathcal{N}}

\newcommand{\re}{\operatorname{Re}}

\newcommand{\res}{\operatorname{Res}}
\newcommand{\BV}{\operatorname{BV}}
\newcommand{\Lip}{\operatorname{Lip}}
\newcommand{\ri}{\mathrm{i}}
\newcommand{\rd}{\mathrm{d}}
\newcommand{\ind}[1]{\mathbf{1}_{#1}}
\newcommand{\rint}{\int\limits}
\newcommand{\rrw}{\rightarrow}

\begin{document}

\title[Representation Functions of Certain Numeration Systems]{On the representation functions of certain\\ numeration systems}

\author{Nian Hong Zhou}
\thanks{This work was partially supported by the Austrian Science Fund (FWF), Grant P 32305 and by the National Natural Science Foundation of China (No.~12661065).}

\address{School of Mathematics and Statistics, Guangxi Normal University, Guilin 541004, P.R. China}
\email{nianhongzhou@outlook.com}

\date{}

\begin{abstract}
Let $\beta>1$ be fixed. We consider the $(\bd)$-numeration system, where the base sequence $\bseq=(b_k)_{k\geq 0}$ consists of positive real numbers satisfying $\lim_{k\rightarrow \infty}b_{k+1}/b_k=\beta$, and the digit set $\dset\ni 0$ is a finite set of nonnegative real numbers containing at least two elements. Let $r_{\bd}(\lambda)$ denote the number of representations of a given $\lambda\in\mathbb{R}$ in the form $\sum_{k\ge 0}\delta_kb_k$ with digits $\delta_k\in\dset$. We establish upper bounds for $r_{\bd}(\lambda)$ as well as asymptotic formulas for its arbitrary moments. Furthermore, we prove that the associated Dirichlet series $\zeta_{\bd}(s):=\sum_{\lambda>0}r_{\bd}(\lambda)\lambda^{-s}$ admits a meromorphic continuation to the entire complex plane when $b_k=\beta^{k}$, and to the half-plane $\re(s)>\log_\beta |\dset|-\gamma$ when $b_k=\beta^{k}+O(\beta^{(1-\gamma)k})$ for any fixed $\gamma\in(0,1]$. Within the regions of continuation, we determine all candidate poles, compute their corresponding residues, and locate the trivial zeros of $\zeta_{\bd}(s)$. As an application, we resolve several problems posed by Chow and Slattery concerning partitions into distinct terms of certain integer sequences.
\end{abstract}


\maketitle

\section{Introduction}

In this paper, we investigate a broad generalization of classical positional numeration systems. Specifically, we consider a $(\bd)$-numeration system defined by a base sequence $\bseq=(b_k)_{k\ge 0}$ and a digit set $\dset$, in which real numbers $\lambda$ are represented as finite sums $\sum_{k\ge 0}\delta_k b_k$ with $\delta_k\in \dset$. Throughout this work, we assume that the base $\bseq=(b_k)_{k\ge 0}$ is a \emph{sequence of positive real numbers satisfying $\lim_{k\to \infty}(b_{k+1}/b_k)=\beta$ for some real constant $\beta>1$}, and that the digit set $\dset\ni 0$ is a \emph{finite set of nonnegative real numbers} with $|\dset|\ge 2$. Here and throughout, it is understood that only finitely many digits $\delta_k$ in any representation $\sum_{k\ge 0}\delta_k b_k$ are nonzero. The set of all real numbers representable in the $(\bd)$-system is denoted by
\begin{equation*}\label{eqda}
\cN_{\bd}
=\left\{\sum_{k\ge 0}\delta_kb_k :(\delta_0, \delta_1,\delta_2,\ldots)\in \dset^{\bn_0} \right\}.
\end{equation*}
We define the representation function $r_{\bd}(\lambda)$ as the number of distinct ways in which $\lambda\in\br$ can be expressed as $\sum_{k\ge 0}\delta_kb_k$ with $\delta_k\in\dset$. Equivalently, $r_{\bd}(\lambda)$ is the cardinality of the set
\begin{equation}\label{eqr}
\cR_{\bd}(\lambda)
=\left\{(\delta_0, \delta_1,\delta_2,\ldots)\in \dset^{\bn_0}:\sum_{k\ge 0}\delta_kb_k=\lambda\right\}.
\end{equation}
For any $x\in\mathbb{R}$, the counting function associated with the $(\bd)$-numeration system is defined by
\begin{equation}\label{eqda1}
S_{\bd}(x)
=\sum_{0\le \lambda \le x}r_{\bd}(\lambda)
=\left|\left\{(\delta_0, \delta_1,\delta_2,\ldots)\in \dset^{\bn_0}:\sum_{k\ge 0}\delta_kb_k\le x\right\}\right|.
\end{equation}
For any scalar $c>0$, writing $c\bseq:=(c b_k)_{k\ge 0}$, one immediately observes the scaling relation 
$$S_{c\bseq, \dset}(cx)=S_{\bd}(x).$$ 
In the special case of the pure geometric base $\bseq=(\beta^k)_{k\ge 0}$, we replace the subscript $\bseq$ with $\beta$ and write $r_{\betaD}(\lambda)$, $S_{\betaD}(x)$, and $\zeta_{\betaD}(s)$.

\medskip

Our first objective is to establish non-trivial upper bounds for $r_{\bd}(\lambda)$ and precise asymptotic formulas for $S_{\bd}(x)$. Most prior studies restrict the digit set to canonical integer intervals $\dset=\{0,1,\ldots,q-1\}$ with $q\ge 2$, rather than arbitrary finite subsets of $\br_{\ge 0}$. In that canonical setting, Dumont, Sidorov, and Thomas \cite[Theorem 2.1]{MR1690457} derived the asymptotic behavior of $S_{\bd}(x)$ when $(b_k)_{k\ge 0}$ satisfies a linear recurrence, while Feng, Liardet, and Thomas \cite[Theorem 14]{MR3237075} treated the case $b_k\sim c\beta^{k}$ ($c>0$, $\beta>1$). Likewise, non-trivial upper bounds for $r_{\bd}(\lambda)$ ($\lambda\in\cN_{\bd}$) have previously been obtained only in special cases: for $b_k=d^k$ ($d\in\bn$) in \cite[Section 6]{MR1690457}; for $b_k\sim c\beta^k$ with $\dset=\{0,1,\ldots,\lfloor\beta\rfloor\}$ in \cite[Proposition 16]{MR3237075} (where $\lfloor x\rfloor$ denotes the greatest integer not exceeding $x$); and for Pisot scales in \cite[Section 3]{MR3237075}. 

A celebrated special case arises when $\bseq=(F_{k+2})_{k\ge 0}$, where $F_k$ is the $k$-th Fibonacci number, and $\dset=\{0,1\}$. Then $R_F(n):=r_{\bd}(n)$ counts the number of representations of $n$ as a sum of distinct Fibonacci numbers, known as the \emph{Fibonacci partition function}. A classical result of Pushkarev \cite{MR1374325} states that $R_{F}(n)=O(\sqrt{n})$, an upper bound that is known to be optimal. Recently, Chow and Slattery \cite{MR4235264} gave an exact formula and tight bounds for the summatory function $S_{\bd}(x)$ of $R_{F}(n)$. They also formulated a conjecture \cite[p.~315]{MR4235264} on the logarithmic average of $R_F(n)$ and posed several open problems concerning: (1) higher moments of $R_{F}(n)$; (2) analogous results for Lucas partitions; and (3) partitions into distinct terms of the sequence $(\lfloor\tau^m\rfloor)_{m\ge 1}$ for a fixed $\tau\in(1,2)$. For further background on numeration systems and partition functions, we refer the reader to \cite{MR1690457,MR3237075,MR4235264} and the references therein.

\medskip

To state our main results, we introduce several key parameters and function spaces. Let
\[
\dmax := \max\dset, \qquad \dmin := \min(\dset\setminus\{0\}), \qquad \dmin^* := \min\{\,|\delta-\delta'| : \delta,\delta'\in\dset,\,\delta\neq\delta'\,\}.
\]
We define $\kappa_{\betaD}$ to be the largest positive real number $u$ such that
\[
\frac{\dmax}{\dmin^*} \sum_{k\ge 1}\beta^{-\lfloor k/u\rfloor}\le 1.
\]
Since the map $u\mapsto \sum_{k\ge 1}\beta^{-\lfloor k/u\rfloor}$ is strictly increasing and left-continuous on $(0,\infty)$, the maximum is well-defined and satisfies $\kappa_{\betaD}\in (0,1]$. When $\beta\in\bn$ and $\dset\subset \bn_0$ with $\gcd(\dset)=1$, we define the maximum residue multiplicity
\[
\mu_{\betaD}:=\max_{a\in \bz}|\{\delta\in\dset: \delta\equiv a\pmod \beta\}|,
\]
which satisfies $\max(1,\beta^{-1}|\dset|)\le \mu_{\betaD}\le |\dset|-1$. 

We identify functions on the torus $\bt:=\br/\bz$ with $1$-periodic complex-valued functions on $\br$. Let $\BV(\bt)$ denote the space of $1$-periodic functions of bounded variation on $\bt$. A function $f:\bt\to\bc$ is called a $1$-periodic Lipschitz (or H\"older) function of order $\eta\in (0,1]$ if
\[
\sup_{t\in\bt,\, \Delta\neq 0}\frac{|f(t+\Delta)-f(t)|}{|\Delta|^\eta}<+\infty,
\]
and we denote the space of such functions by $\Lip_\eta(\bt)$. Finally, for any $x\in\br$, we define the periodic profile function
\begin{align}\label{eqpsi}
\Psi_{\betaD}(x)
=\frac{S_{\betaD}(\beta^x)}{|\dset|^{x}}-\frac{1}{|\dset|}\sum_{h\ge 0}\frac{1}{|\dset|^{h+x}}
\sum_{\delta\in \dset}\left(S_{\betaD}(\beta^{h+x})-S_{\betaD}(\beta^{h+x}-\beta^{-1}\delta)\right).
\end{align}
We begin by summarizing the fundamental analytic properties of $\Psi_{\betaD}(x)$, which will be proved in Section \ref{sec22}.

\begin{proposition}\label{lemm20}
The series in \eqref{eqpsi} converges absolutely for all $x\in\mathbb{R}$. For any $\beta>1$, the function $\Psi_{\betaD}(x)$ is strictly positive and belongs to $\BV(\bt)\cap \Lip_{\eta}(\bt)$ for every $\eta\in (0, \kappa_{\betaD}\log_{\beta}|\dset|)$. Moreover, if $\beta\in\bn$ and $\dset\subset \bn_0$ with $\gcd(\dset)=1$, then $\Psi_{\betaD}(x)\in \Lip_{\eta}(\bt)$ with $\eta=\log_\beta|\dset|-\log_\beta\mu_{\betaD}$.
\end{proposition}

By a classical theorem of Zygmund (see \cite[Theorem (3.6), p.~241]{MR1963498} or \cite[p.~35]{MR2039503}), any function in $\BV(\bt)\cap \Lip_{\eta}(\bt)$ ($\eta>0$) possesses an absolutely convergent Fourier series. In Section \ref{sec31}, we establish the following explicit representation for the Fourier coefficients of $\Psi_{\betaD}(x)$.

\begin{proposition}\label{pro36}
For any $x\in\br$, we have the absolutely convergent Fourier expansion 
$$\Psi_{\betaD}(x)
=\sum_{k\in\bz}\widehat{\Psi}_{\betaD}(k)e^{2\pi\ri k x},$$ 
where
\[
\widehat{\Psi}_{\betaD}(k)=\frac{(-1)^k\sqrt{|\dset|}}{\Gamma\left(1+\log_\beta|\dset|+2\pi (\log \beta)^{-1}k\ri\right)}\rint_{\bt}\prod_{h\in\bz}\left(\frac{\sum_{\delta\in \dset}\exp(-\delta \beta^{h+w-1/2})}{
\sum_{\delta\in \dset}\exp(-\delta \beta^{h+w+1/2})}\right)^{h+w}e^{2\pi\ri k w}\,\rd w.
\]
\end{proposition}

Our first main theorems concern upper bounds for $r_{\bd}(\lambda)$ and asymptotic formulas for $S_{\bd}(x)$. Their proofs, along with that of Corollary \ref{cor1}, are given in Section \ref{sec2}.

\begin{theorem}\label{th2}
Suppose that $\bseq=(\beta^k)_{k\ge 0}$ with $\beta\in\bn$, and let $\dset\subset \bn_0$ satisfy $\gcd(\dset)=1$. Then for every $n\in\bn$,
\begin{equation}\label{eqrcs}
r_{\betaD}(n)\le \mu_{\betaD}\, n^{\log_\beta \mu_{\betaD}},
\end{equation}
and for all $x>0$,
\begin{equation}\label{eqdas1s}
0\le S_{\betaD}(x)-x^{\log_\beta|\dset|}\Psi_{\betaD}(\log_\beta x)\le x^{\log_\beta \mu_{\betaD}} \frac{\mu_{\betaD}}{|\dset|-\mu_{\betaD}} \sum_{\delta\in \dset}\left(1+\left\lfloor\frac{\delta}{\beta}\right\rfloor\right).
\end{equation}
\end{theorem}

Theorem \ref{th2} addresses the integer $\beta$-adic base $\bseq=(\beta^k)_{k\ge 0}$ with a general integer digit set $\dset\subset\bn_0$. For general real bases $\bseq$ and real digit sets $\dset$, we establish the following result.

\begin{theorem}\label{th1}
For any $\Delta>0$ and $\varepsilon>0$, we have
\begin{align}\label{eqds1}
S_{\bd}(x)-S_{\bd}(x-\Delta)\le (1+\Delta)x^{(1-\kappa_{\betaD})\log_\beta|\dset|+\varepsilon}
\end{align}
for all sufficiently large $x$. Furthermore, for every $x\in\br$, the scaling limit
\begin{align}\label{eqdas1}
\lim_{n\to+\infty} |\dset|^{-n-x}S_{\bd}(b_n\beta^x)=\Psi_{\betaD}(x)
\end{align}
holds. If, in addition, $b_k=\beta^{k}+O(\beta^{(1-\gamma) k})$ for some fixed $\gamma\in (0, 1]$, then as $x\rrw+\infty$,
\begin{align}\label{prosssc}
x^{-\log_\beta |\dset|}S_{\bd}(x)=\Psi_{\betaD}(\log_\beta x)+O\left(x^{-\gamma\kappa_{\betaD} \log_{\beta}|\dset|+\varepsilon}\right).
\end{align}
\end{theorem}

\begin{remark}
Numerical illustrations of the scaling limit \eqref{eqdas1} are displayed in Figure \ref{fig1}. Note that if $\bseq=(\beta^k)_{k\ge 0}$ with $\beta\in\bn$, and $\dset\subset\bn_0$ satisfies $\gcd(\dset)=1$, $|\dset|=\beta$, and $\mu_{\betaD}=1$, then Theorem \ref{th2} implies $r_{\betaD}(n)\in\{0,1\}$ for all $n\in\bn_0$. In this case, the asymptotic density of $\cN_{\betaD}$ fails to exist whenever
\[
\liminf_{x\rrw +\infty}\frac{S_{\betaD}(x)}{x}= \liminf_{x\rrw +\infty}\Psi_{\betaD}(\log_\beta x)\neq \limsup_{x\rrw +\infty}\Psi_{\betaD}(\log_\beta x)= \limsup_{x\rrw +\infty}\frac{S_{\betaD}(x)}{x}.
\]
Since $\Psi_{\betaD}$ is generically non-constant (see Proposition \ref{pro36}), it naturally plays the role of a logarithmic density profile. For the general $(\bd)$-numeration system in Theorem \ref{th1}, $\log_\beta |\dset|$ need not equal $1$ and $\cN_{\bd}$ need not be a subset of $\bn_0$; we therefore refer to $\Psi_{\betaD}(x)$ as the \emph{relative density function} of the $(\bd)$-numeration system.

\begin{figure}[htbp]
\begin{subfigure}{.45\columnwidth}
\centering
\includegraphics[width=1\columnwidth]{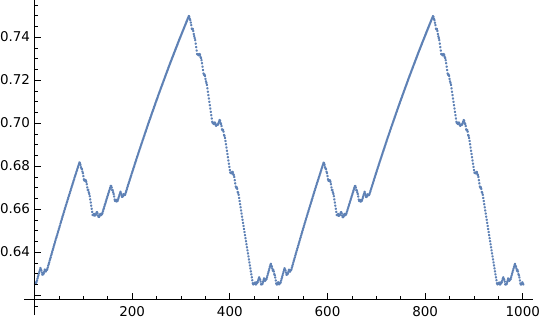}
\caption{Values of $3^{-x}S_{3,\dset}(3^x)$}
\label{fig:sub1}
\end{subfigure}
\begin{subfigure}{.45\columnwidth}
\centering
\includegraphics[width=1\columnwidth]{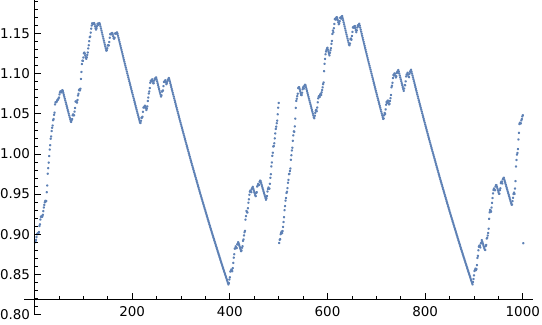}
\caption{Values of $3^{-x}S_{\bd}(3^{x-\lfloor x\rfloor}b_{\lfloor x\rfloor})$}
\label{fig:sub2}
\end{subfigure}%
\caption{Numerical plots of $|\dset|^{-x}S_{\bd}(\beta^{x-\lfloor x\rfloor}b_{\lfloor x\rfloor})$}
\label{fig1}
\medskip
{\small In Figure \ref{fig:sub1}, $(\bseq,\dset)=\left((3^k)_{k\ge 0},\{0,1,5\}\right)$ for $x=8+k/500$. In Figure \ref{fig:sub2}, $(\bseq,\dset)=\left((\binom{2k}{k})_{k\ge 0},\{0,1,3\}\right)$ for $x=6+k/500$. Both plots closely match the theoretical prediction of Theorem \ref{th1}.}
\end{figure}
\end{remark}

As a consequence of the scaling limit \eqref{eqdas1} in Theorem \ref{th1}, we obtain the following asymptotic formulas for the moments of $r_{\bd}(\lambda)$.

\begin{corollary}\label{cor1}
Let $x\in\br$ be fixed. For every $k>0$, as $n\rrw+\infty$,
\begin{equation}\label{eqrekm}
\frac{|\dset|^{\log_\beta b_n-n}}{(\beta^{x}b_n)^k}\sum_{0< \lambda\le \beta^{x}b_n}\frac{r_{\bd}(\lambda)}{\lambda^{\log_\beta |\dset|-k}}=\Psi_{\betaD}(x)-\frac{\log(\beta^k/ |\dset|)}{\beta^{k}-1}\int_{0}^{1}\Psi_{\betaD}(x+w)\beta^{kw}\,\rd w+o(1).
\end{equation}
At the critical exponent $k=0$, as $n\rrw+\infty$,
\begin{equation}\label{eqrelog}
\frac{1}{\log_\beta|\dset|}\sum_{0< \lambda\le \beta^{x}b_n}\frac{r_{\bd}(\lambda)}{\lambda^{\log_\beta |\dset|}}=(1+o(1))\sum_{0\le h< n}|\dset|^{h-\log_\beta b_h}\int_{0}^1\Psi_{\betaD}(u)\,\rd u+O(1),
\end{equation}
whereas for every $k<0$, the sum $\sum_{\lambda>0}r_{\bd}(\lambda)\lambda^{k-\log_\beta|\dset|}$ converges to a finite constant.
\end{corollary}

Observe that the Fibonacci numbers $F_n$, the Lucas numbers $L_n$, and the sequence $\lfloor\tau^n\rfloor$ ($1<\tau<2$) all satisfy $b_n\sim c\beta^n$ as $n\rrw+\infty$ for suitable constants $\beta>1$ and $c>0$. Specializing Corollary \ref{cor1} to $b_n\sim c\beta^n$ and $\dset=\{0,1\}$ resolves Conjecture 1.4 of Chow and Slattery \cite{MR4235264} on the logarithmic average of the Fibonacci partition function $R_{F}(n)$, as well as their three open problems \cite[p.~315, (1)--(3)]{MR4235264}. Specifically, using the scaling property $S_{c\bseq, \dset}(cx)=S_{\bd}(x)$, \eqref{eqrelog} yields
\begin{equation*}
\frac{1}{\log x}\sum_{0< \lambda\le x}\frac{r_{\bd}(\lambda)}{\lambda^{\log_\beta |\dset|}}=\frac{\log_\beta |\dset|}{c^{\log_\beta|\dset|}\log\beta}\int_{0}^1\Psi_{\betaD}(u)\,\rd u+o(1),
\end{equation*}
while for each $k>0$, \eqref{eqrekm} gives
\begin{equation*}
\frac{1}{(cx)^k}\sum_{0< \lambda\le cx}\frac{r_{\bd}(\lambda)}{\lambda^{\log_\beta |\dset|-k}}=\Psi_{\betaD}(\log_\beta x)-\frac{\log(\beta^k/|\dset|)}{\beta^{k}-1}\int_{0}^{1}\Psi_{\betaD}(w+\log_\beta x)\beta^{kw}\,\rd w+o(1)
\end{equation*}
as $x\rrw+\infty$.

\medskip

Next, we turn to the Dirichlet series (or zeta function) associated with the $(\bd)$-numeration system, defined for $\re(s)$ sufficiently large by
\begin{equation}\label{eqzeta}
\zeta_{\bd}(s)=\sum_{\lambda>0}\frac{r_{\bd}(\lambda)}{\lambda^{s}}.
\end{equation}
As a direct consequence of Corollary \ref{cor1}, we determine the abscissa of convergence of $\zeta_{\bd}(s)$.

\begin{proposition}
The abscissa of convergence of $\zeta_{\bd}(s)$ is $\sigma_c = \log_\beta |\dset|$. Moreover, the series $\zeta_{\bd}(\log_\beta |\dset|)$ converges if and only if $\sum_{h\geq 0}|\dset|^{h-\log_\beta b_h}<+\infty$.
\end{proposition}

When $\bseq=(\beta^k)_{k\ge 0}$ with $\beta\in \bn$ and $\dset=\{0,1,\ldots,\beta-1\}\setminus\{u\}$ for some integer $u\in [1, \beta)$, $r_{\betaD}(n)$ is the indicator function of positive integers whose $\beta$-adic expansion omits the digit $u$. Hardy and Wright \cite[Theorem 144]{MR2445243} proved that $\zeta_{\betaD}(1)=\sum_{n\ge 1}r_{\betaD}(n)/n$ converges (the Kempner series). Nathanson \cite[Corollary 2]{MR4215805} strengthened this by showing that $\zeta_{\betaD}(s)$ has abscissa of convergence $\sigma_c=\log_\beta(\beta-1)$, and further extended this result to digit sets varying across positions under suitable asymptotic conditions.

\medskip

Our second main objective is to establish the meromorphic continuation of $\zeta_{\bd}(s)$. When $\beta\in\bn$ and $\dset=\{0,1,\ldots,\beta-1\}$, $\zeta_{\betaD}(s)$ reduces to the classical Riemann zeta function $\zeta(s)$, which extends meromorphically to $\bc$ with a single simple pole at $s=1$ of residue $1$, and takes the values $\zeta(-n)=(-1)^nB_{n+1}/(n+1)$ at non-positive integers $n\ge 0$, where $B_n$ denotes the $n$-th Bernoulli number. Our next theorem generalizes these classical properties to $\zeta_{\bd}(s)$. Its proof is given in Section \ref{sec3}.

\begin{theorem}\label{th3}
The zeta function $\zeta_{\bd}(s)$ admits a meromorphic continuation to the half-plane $\re(s)>\log_\beta |\dset|-\gamma$ when $b_k=\beta^{k}+O(\beta^{(1-\gamma)k})$ for any fixed $\gamma\in(0,1]$, and to the entire complex plane $\bc$ when $b_k=\beta^{k}$. Within its domain of continuation, $\zeta_{\bd}(s)$ is holomorphic except for possible simple poles located at the points
\[
p_{\betaD}(j,k)=\log_\beta |\dset|-j-\frac{2\pi\ri k}{\log\beta} \qquad (j\in \bn_0,\, k\in\bz).
\]
For $j=0$, the residue at $s=p_{\betaD}(0,k)$ in both cases is given by
\[
\res_{s=p_{\betaD}(0,k)}\zeta_{\bd}(s)=\left(\log_\beta|\dset|-\frac{2\pi\ri k}{\log\beta}\right)\widehat{\Psi}_{\betaD}(-k).
\]
Furthermore, when $b_k=\beta^k$, the residue of $\zeta_{\betaD}(s)$ at $s=p_{\betaD}(j,k)$ for arbitrary $j\in\bn_0$ and $k\in\bz$ is given by
\[
\res_{s=p_{\betaD}(j,k)}\zeta_{\betaD}(s)=c_{\betaD}(j)\frac{\Gamma\left(1+\log_\beta|\dset|-2\pi(\log\beta)^{-1}k\ri\right)}{\Gamma\left(\log_\beta|\dset|-j-2\pi(\log\beta)^{-1}k\ri\right)}\widehat{\Psi}_{\betaD}(-k),
\]
where $\widehat{\Psi}_{\betaD}(-k)$ is the Fourier coefficient given in Proposition \ref{pro36}, and
\begin{equation*}
c_{\betaD}(j)=\sum_{\substack{\ell_1,\ell_2,\ldots,\ell_j\ge 0\\ 1\cdot \ell_1+2\cdot \ell_2+\cdots+j \ell_j=j}}\prod_{1\le h\le j}\frac{1}{\ell_h!}\left(\frac{[y^h]L_{\dset}(y)}{1-\beta^h}\right)^{\ell_h},
\end{equation*}
with $[y^h]L_{\dset}(y)$ denoting the coefficient of $y^h$ in the Maclaurin expansion of 
$$L_{\dset}(y):=\log\left(\sum_{\delta\in \dset}e^{-\delta y}\right).$$ 
Moreover, when $b_k=\beta^k$, setting $c_{\betaD}(u)=0$ for $u\notin\bn_0$, we have for every $n\in\bn_0$,
\begin{align*}
\zeta_{\betaD}(-n)=(-1)^n n!\, c_{\betaD}(n+\log_\beta |\dset|)\,\Gamma(1+\log_\beta|\dset|)\widehat{\Psi}_{\betaD}(0)-\ind{n=0},
\end{align*}
where $\ind{A}$ denotes the indicator function of the condition $A$.
\end{theorem}

\section{Upper bounds and moments of \texorpdfstring{$r_{\bd}(\lambda)$}{r(lambda)}}\label{sec2}

In this section, we establish upper bounds, asymptotic averages, and moment formulas for $r_{\bd}(\lambda)$, thereby proving Proposition \ref{lemm20}, Theorems \ref{th2} and \ref{th1}, and Corollary \ref{cor1}.

\subsection{Upper bounds for \texorpdfstring{$r_{\bd}(\lambda)$}{r(lambda)}}\label{sec21}

We begin by proving the pointwise upper bound \eqref{eqrcs} of Theorem \ref{th2}.

\begin{proof}[Proof of \eqref{eqrcs}]
Recall that $\bseq=(\beta^k)_{k\ge 0}$ with $\beta\in\bn$ and $\dset\subset \bn_0$ with $\gcd(\dset)=1$. Since $\beta^k>n$ for all integers $k>m:=\lfloor\log_\beta n\rfloor$, we have
\[
\cR_{\betaD}(n)=\left\{(\delta_0, \delta_1,\ldots,\delta_{m})\in \dset^{1+m}:
\sum_{0\le k\le m}\delta_k\beta^k=n
\right\}.
\]
Let $n=\sum_{0\le k\le m} d_k\beta^k$ with $d_k\in\{0,1,\ldots,\beta-1\}$ be the standard $\beta$-adic expansion of $n$. For any tuple $(\delta_0, \delta_1,\ldots,\delta_{m})\in \cR_{\betaD}(n)$ and any $0\le j\le m$, reducing modulo $\beta^{j+1}$ gives
\[
\sum_{0\le k\le j} \delta_k\beta^k\equiv \sum_{0\le k\le j} d_k\beta^k\pmod {\beta^{j+1}}.
\]
Consequently, $\beta^{-j}\big(\sum_{0\le k< j} d_k\beta^k-\sum_{0\le k< j} \delta_k\beta^k\big)$ is an integer, and dividing by $\beta^j$ yields the congruence
\[
\delta_j\equiv d_j+\beta^{-j}\left(\sum_{0\le k< j} d_k\beta^k-\sum_{0\le k< j} \delta_k\beta^k\right)\pmod {\beta}.
\]
Thus, once the preceding digits $\delta_0,\delta_1,\ldots,\delta_{j-1}$ are fixed, the number of admissible choices for $\delta_j\in\dset$ is at most
\[
\max_{a\in\bz}|\{\delta\in\dset: \delta\equiv a\pmod \beta\}|= \mu_{\betaD}.
\]
Iterating over $0\le j\le m$, we conclude that
\[
r_{\betaD}(n)=|\cR_{\betaD}(n)|\le \mu_{\betaD}^{1+m}\le \mu_{\betaD}^{1+\log_\beta n}=\mu_{\betaD}\,n^{\log_\beta \mu_{\betaD}},
\]
which completes the proof of \eqref{eqrcs}.
\end{proof}

We now turn to the general setting of Theorem \ref{th1} and prove the local bound \eqref{eqds1}. To this end, we require two auxiliary lemmas.

\begin{lemma}\label{lem21}
For any real number $u\in (0, \kappa_{\betaD})$, let $\cseq_u=(b_{j_k(u)})_{k\ge 0}$ be the subsequence of $\bseq$ indexed by $j_k(u)=\lfloor k/u\rfloor$. Then for all $x\in\br$,
\[
S_{\cseq_u,\dset}(x+1)-S_{\cseq_u,\dset}(x-1)=O_u(1).
\]
In particular, $r_{\cseq_u,\dset}(x)=O_u(1)$.
\end{lemma}

\begin{proof}
Since $\lim_{k\rightarrow+\infty}b_{k+1}/b_k=\beta$, applying the reverse Fatou lemma to any tuple $(\delta_0,\delta_1,\ldots,\delta_{n-1})\in \dset^n$ gives
\begin{align*}
\limsup_{n\rrw \infty}\left(\frac{1}{\dmin^* b_{j_n(u)}-2}\sum_{1\le k\le n}\delta_{n-k}b_{j_{n-k}(u)}\right)&\le \frac{1}{\dmin^*}\sum_{k\ge 1}\limsup_{n\rrw \infty}\left(\delta_{n-k}\frac{b_{j_{n-k}(u)}}{b_{j_n(u)}}\ind{k<n}\right)\\
&\le \frac{\dmax}{\dmin^*} \sum_{k\ge 1}\limsup_{n\rrw \infty}\left(\beta^{j_{n-k}(u)-j_{n}(u)}\ind{k<n}\right)\\
&\le \frac{\dmax}{\dmin^*} \sum_{k\ge 1}\beta^{-\lfloor k/u\rfloor}<1,
\end{align*}
where we used the inequality $j_{n-k}(u)-j_{n}(u)=\lfloor(n-k)/u\rfloor-\lfloor n/u\rfloor\le -\lfloor k/u\rfloor$ together with the strict monotonicity of $v\mapsto \sum_{k\ge 1}\beta^{-\lfloor k/v\rfloor}$ for $0<u<\kappa_{\betaD}$. Consequently, there exists an integer $n_{u}>0$ such that for all $n\geq n_{u}$ and all $(\delta_0,\ldots,\delta_{n-1})\in \dset^n$,
\[
\sum_{0\le k< n}\delta_kb_{j_{k}(u)}<\dmin^* b_{j_n(u)}-2.
\]
Suppose now that two finitely supported sequences $(\delta_k)_{k\geq 0}, (\delta_k')_{k\geq 0}\in \dset^{\mathbb{N}_0}$ satisfy
\[
\left|\sum_{k\ge 0}\delta_kb_{j_k(u)}-\sum_{k\ge 0}\delta_k'b_{j_k(u)}\right|\le 2.
\]
We claim that $\delta_k=\delta_k'$ for all $k\geq n_u$. Indeed, if this were false, let $k_m\ge n_u$ be the largest index where $\delta_{k_m}\neq \delta_{k_m}'$, so that $\delta_k=\delta_k'$ for all $k>k_m$. Assuming without loss of generality that $\delta_{k_m}>\delta_{k_m}'$, we have $\delta_{k_m}-\delta_{k_m}'\ge \dmin^*$, and therefore
\[
\left|\sum_{k\ge 0}\delta_kb_{j_k(u)}-\sum_{k\ge 0}\delta_k'b_{j_k(u)}\right|
=\left|\sum_{0\le k\le k_m}(\delta_k-\delta_k')b_{j_k(u)}\right|
\ge \dmin^* b_{j_{k_m}(u)}-\sum_{0\le k<k_m}\delta_k'b_{j_k(u)}>2,
\]
a contradiction. Thus $\delta_k=\delta_k'$ for all $k\geq n_u$.

In particular, for any $x\in\br$, if there exists $\lambda=\sum_{k\ge 0}\delta_kb_{j_k(u)}\in \cN_{\cseq_u,\dset}$ such that $|\lambda-x|\le 1$, then by the monotonicity of $S_{\cseq_u,\dset}$,
\begin{align*}
S_{\cseq_u,\dset}(x+1)-S_{\cseq_u,\dset}(x-1)&\le S_{\cseq_u,\dset}(\lambda+2)-S_{\cseq_u,\dset}(\lambda-2)\\
&=\sum_{\lambda-2<\lambda'\le \lambda+2}r_{\cseq_u,\dset}(\lambda')\\
&\le\left|\left\{(\delta_k')_{k\geq 0}\in \dset^{\mathbb{N}_0}:
\left|\sum_{k\ge 0}\delta_kb_{j_k(u)}-\sum_{k\ge 0}\delta_k'b_{j_k(u)}\right|\le 2 \right\}\right|\\
&\le |\dset|^{n_u}=O_u(1).
\end{align*}
If no such $\lambda$ exists in $[x-1,x+1]$, then $S_{\cseq_u,\dset}(x+1)-S_{\cseq_u,\dset}(x-1)=0$. This completes the proof.
\end{proof}

\begin{lemma}\label{lemm22}
For any $\varepsilon>0$, we have
\[
S_{\bd}(\lambda+1)-S_{\bd}(\lambda-1)\le \lambda^{(1-\kappa_{\betaD})\log_\beta |\dset|+\varepsilon}
\]
for all sufficiently large $\lambda\in \cN_{\bd}$.
\end{lemma}

\begin{proof}
Fix any $u\in (0, \kappa_{\betaD})$ and let $\aseq_u=(a_k)_{k\ge 0}$ be the complementary subsequence obtained by removing the terms of $\cseq_u$ (defined in Lemma \ref{lem21}) from $\bseq$. Decomposing representations over $\bseq$ into parts from $\aseq_u$ and $\cseq_u$, and applying Lemma \ref{lem21}, we find
\begin{align*}
S_{\bd}(\lambda+1)-S_{\bd}(\lambda-1)&=\sum_{\substack{\lambda_1,\lambda_2\ge 0\\ \lambda-1<\lambda_1+\lambda_2\le \lambda+1}}r_{\aseq_u,\dset}(\lambda_1)r_{\cseq_u,\dset}(\lambda_2)\\
&\le\sum_{0\le \lambda_1\le \lambda+1}r_{\aseq_u,\dset}(\lambda_1) \sum_{\lambda-\lambda_1-1<\lambda_2\le \lambda-\lambda_1+1}r_{\cseq_u,\dset}(\lambda_2)\\
&\ll_u\sum_{0\le \lambda_1\le \lambda+1}r_{\aseq_u,\dset}(\lambda_1).
\end{align*}
Since every nonzero digit in $\dset$ is bounded below by $\dmin>0$, any representation $\sum_{k\ge 0}\delta_k a_k\le \lambda+1$ can only use nonzero digits at indices $k$ for which $\dmin a_k\le \lambda+1$. Hence
\begin{align*}
S_{\bd}(\lambda+1)-S_{\bd}(\lambda-1)&\ll_u \left|\left\{(\delta_0, \delta_1,\delta_2,\ldots)\in \dset^{\bn_0}:
\sum_{k\ge 0}\delta_ka_k\le \lambda+1
\right\}\right|\\
&\le |\dset|^{1+\max\{k:\, \dmin a_k\le \lambda+1\}}.
\end{align*}
Since $\lim_{k\to\infty}b_{k+1}/b_k=\beta$, we have $\log_\beta b_{k}=k(1+o(1))$ as $k\to\infty$. Consequently, as $\lambda\to+\infty$,
\begin{align*}
\max\{k: \dmin a_k\le \lambda+1\}&=\max\{k: \dmin b_k\le \lambda+1\}-\max\{k: \dmin b_{j_k(u)}\le \lambda+1\}\\
&= (1-u)\log_\beta \lambda+o(\log \lambda).
\end{align*}
It follows that as $\lambda\rrw+\infty$,
\[
S_{\bd}(\lambda+1)-S_{\bd}(\lambda-1)\ll_u |\dset|^{(1-u+o(1))\log_\beta \lambda}=\lambda^{(1-u+o(1))\log_\beta |\dset|}.
\]
Since $u\in (0,\kappa_{\betaD})$ can be chosen arbitrarily close to $\kappa_{\betaD}$, the lemma follows.
\end{proof}

\begin{proof}[Proof of \eqref{eqds1}]
For all sufficiently large $x$ and any integer $0\le k\le \Delta$, if there exists $\lambda\in\cN_{\bd}\cap (x-k-1, x-k]$, then by the monotonicity of $S_{\bd}$ and Lemma \ref{lemm22},
\begin{align*}
S_{\bd}(x-k)-S_{\bd}(x-k-1)&\le S_{\bd}(\lambda+1)-S_{\bd}(\lambda-1)\\
&\le \lambda^{(1-\kappa_{\betaD})\log_\beta |\dset|+\varepsilon}\le x^{(1-\kappa_{\betaD})\log_\beta |\dset|+\varepsilon}.
\end{align*}
If no such $\lambda$ exists, the difference is zero. Summing over $0\le k\le \lfloor\Delta\rfloor$ gives
\begin{align*}
S_{\bd}(x)-S_{\bd}(x-\Delta)&\le \sum_{0\le k\le \Delta}\left(S_{\bd}(x-k)-S_{\bd}(x-k-1)\right)\\
&\le (1+\Delta) x^{(1-\kappa_{\betaD})\log_\beta |\dset|+\varepsilon},
\end{align*}
as claimed.
\end{proof}

\subsection{Asymptotic formulas for \texorpdfstring{$S_{\bd}(x)$}{S(x)}}\label{sec22}

In this subsection, we study the zeroth moment $S_{\bd}(x)$ and complete the proofs of Proposition \ref{lemm20} and Theorems \ref{th2} and \ref{th1}.

\begin{lemma}\label{lemm21}
For any $n\in\bn$ and $x\in\br$, we have
\begin{align}\label{eqsrc}
\frac{S_{\betaD}(\beta^{n+x})}{|\dset|^{n+x}}=\frac{S_{\betaD}(\beta^x)}{|\dset|^{x}}-\frac{1}{|\dset|}\sum_{0\le h<n}\frac{1}{|\dset|^{h+x}}\sum_{\delta\in \dset}\left(S_{\betaD}(\beta^{h+x})-S_{\betaD}(\beta^{h+x}-\beta^{-1}\delta)\right).
\end{align}
Moreover, the series \eqref{eqpsi} defining $\Psi_{\betaD}(x)$ converges absolutely for all $x\in\br$, and for any $\kappa<\kappa_{\betaD}$,
\[
|\dset|^{-n-x}S_{\betaD}(\beta^{n+x})=\Psi_{\betaD}(x)+O_x(|\dset|^{-\kappa n}) \qquad (n\rrw \infty,\, n\in\bn).
\]
\end{lemma}

\begin{proof}
By partitioning representations according to their initial digit $\delta_0=\delta\in\dset$, we obtain for any $h\in\bn$,
\begin{align*}
S_{\betaD}(\beta^{h+x})&=\left|\left\{(\delta_0,\delta_1,\delta_2,\ldots)\in \dset^{\bn_0}:\sum_{k\ge 0}\delta_k \beta^k\le \beta^{h+x}\right\}\right|\\
&=\sum_{\delta\in \dset}\left|\left\{(\delta_1,\delta_2,\ldots)\in \dset^{\bn}:\sum_{k\ge 0}\delta_{k+1} \beta^k\le \beta^{h-1+x}-\frac{\delta}{\beta}\right\}\right|\\
&=|\dset|S_{\betaD}(\beta^{h-1+x})-\sum_{\delta\in \dset}\left(S_{\betaD}(\beta^{h-1+x})-S_{\betaD}(\beta^{h-1+x}-\beta^{-1}\delta)\right).
\end{align*}
Dividing both sides by $|\dset|^{h+x}$ yields the telescoping relation
\begin{align*}
\frac{S_{\betaD}(\beta^{h+x})}{|\dset|^{h+x}}-\frac{S_{\betaD}(\beta^{h-1+x})}{|\dset|^{h-1+x}}=-\frac{1}{|\dset|^{h+x}}\sum_{\delta\in \dset}\left(S_{\betaD}(\beta^{h-1+x})-S_{\betaD}(\beta^{h-1+x}-\beta^{-1}\delta)\right).
\end{align*}
Summing over $h=1,\ldots,n$ establishes \eqref{eqsrc}. Furthermore, applying the upper bound \eqref{eqds1} of Theorem \ref{th1}, we see that for any $\varepsilon>0$,
\begin{align*}
\sum_{k\ge n}\frac{1}{|\dset|^{1+k+x}}\sum_{\delta\in \dset}\left(S_{\betaD}(\beta^{k+x})-S_{\betaD}(\beta^{k+x}-\beta^{-1}\delta)\right)&\ll_x \sum_{k\ge n}\frac{\beta^{k((1-\kappa_{\betaD})\log_\beta |\dset|+\varepsilon)}}{|\dset|^{k}}\\
&\ll_x |\dset|^{-n(\kappa_{\betaD}-\varepsilon/\log_\beta|\dset|)}.
\end{align*}
Thus the series \eqref{eqpsi} converges absolutely and the stated error bound follows.
\end{proof}

\begin{proof}[Proof of \eqref{eqdas1s}]
Under the hypotheses of Theorem \ref{th2}, we have $\cN_{\betaD}\subseteq\bn_0$. By \eqref{eqsrc} of Lemma \ref{lemm21},
\begin{align*}
\frac{S_{\betaD}(\beta^{n+x})}{|\dset|^{n+x}}&=\frac{S_{\betaD}(\beta^{x})}{|\dset|^{x}}-\frac{1}{|\dset|}\sum_{0\le h<n}\frac{1}{|\dset|^{h+x}}\sum_{\delta\in \dset}\sum_{\beta^{h+x}-\beta^{-1}\delta<\ell\le \beta^{h+x}}r_{\betaD}(\ell).
\end{align*}
Applying the pointwise bound \eqref{eqrcs} to $r_{\betaD}(\ell)$, we have for any $x\in\br$ and $h\ge 0$,
\begin{align*}
0\le \sum_{\delta\in \dset}\sum_{\beta^{h+x}-\beta^{-1}\delta<\ell\le \beta^{h+x}}r_{\betaD}(\ell)&\le \mu_{\betaD}\beta^{(h+x)\log_\beta \mu_{\betaD}}\sum_{\delta\in \dset\setminus\{0\}}\left(\lfloor\beta^{h+x}\rfloor-\lceil \beta^{h+x}-\beta^{-1}\delta\rceil+1\right),
\end{align*}
where $\lceil u\rceil$ denotes the least integer greater than or equal to $u$. Since
\[
\lfloor\beta^{h+x}\rfloor-\lceil \beta^{h+x}-\beta^{-1}\delta\rceil=\lfloor\delta/\beta\rfloor-\lceil \{\beta^{h+x}\}-\{\delta/\beta\}\rceil\le \lfloor\delta/\beta\rfloor,
\]
where $\{u\}=u-\lfloor u\rfloor$, and $1\le \mu_{\betaD}\le |\dset|-1$, we obtain
\begin{align*}
0\le S_{\betaD}(\beta^{n+x})-|\dset|^{n+x}\Psi_{\betaD}(x)&\le \frac{1}{|\dset|}\sum_{h\ge n}\frac{|\dset|^{n+x}}{|\dset|^{h+x}}\sum_{\delta\in \dset}\sum_{\beta^{h+x}-\beta^{-1}\delta<\ell\le \beta^{h+x}}r_{\betaD}(\ell)\\
&\le \frac{1}{|\dset|}\sum_{h\ge 0}\frac{\mu_{\betaD}^{1+n+h+x}}{|\dset|^{h}}\sum_{\delta\in \dset\setminus\{0\}}\left(1+\lfloor\delta/\beta\rfloor\right)\\
&=\frac{\mu_{\betaD}^{1+n+x}}{|\dset|-\mu_{\betaD}} \sum_{\delta\in \dset\setminus\{0\}}\left(1+\left\lfloor\frac{\delta}{\beta}\right\rfloor\right).
\end{align*}
Substituting $\beta^{n+x}\mapsto x$ completes the proof of \eqref{eqdas1s}.
\end{proof}

\begin{proof}[Proof of Proposition \ref{lemm20}]
Absolute convergence of \eqref{eqpsi} was established in Lemma \ref{lemm21}. To verify $1$-periodicity, observe from \eqref{eqpsi} that
\begin{align*}
\Psi_{\betaD}(x+1)=|\dset|^{-x-1}S_{\betaD}(\beta^{1+x})-\frac{1}{|\dset|}\sum_{h\ge 1}\frac{1}{|\dset|^{h+x}}\sum_{\delta\in \dset}\left(S_{\betaD}(\beta^{h+x})-S_{\betaD}(\beta^{h+x}-\beta^{-1}\delta)\right).
\end{align*}
Substituting \eqref{eqsrc} with $n=1$ into the first term on the right-hand side yields $\Psi_{\betaD}(x+1)=\Psi_{\betaD}(x)$.

Next, we establish the Lipschitz regularity of $\Psi_{\betaD}(x)$. For $x, x+\Delta\in[0,2]$ and an integer $n=-\lfloor\log_\beta |\Delta|\rfloor$ as $|\Delta|\rrw 0^+$, Lemma \ref{lemm21} and \eqref{eqds1} of Theorem \ref{th1} give, for any $\kappa<\kappa_{\betaD}$,
\begin{align*}
\Psi_{\betaD}(x+\Delta)-\Psi_{\betaD}(x)&=\left(|\dset|^{-n-x-\Delta}S_{\betaD}(\beta^{n+x+\Delta})-|\dset|^{-n-x}S_{\betaD}(\beta^{n+x})\right)+O\left(|\dset|^{-\kappa n}\right)\\
&\ll\frac{S_{\betaD}(\beta^{n+2})}{|\dset|^{n+1}}\left||\dset|^{-\Delta}-1\right| + \frac{\left|S_{\betaD}(\beta^{n+x+\Delta})-S_{\betaD}(\beta^{n+x})\right|}{|\dset|^{n}}+|\dset|^{-\kappa n}\\
&\ll |\Delta|+(1+\beta^{n+x}|\beta^{\Delta}-1|)|\dset|^{-\kappa n}+|\dset|^{-\kappa n}.
\end{align*}
Since $\beta^n|\Delta|\asymp 1$, we obtain
\begin{align}\label{eq1111}
\Psi_{\betaD}(x+\Delta)-\Psi_{\betaD}(x)
\ll |\Delta|+|\dset|^{\kappa\log_\beta |\Delta|}\ll |\Delta|^{\kappa\log_\beta |\dset|}.
\end{align}
By periodicity, $\Psi_{\betaD}\in \Lip_{\eta}(\bt)$ for every $\eta<\kappa_{\betaD}\log_\beta |\dset|$. In the special case where $\bseq=(\beta^k)_{k\ge 0}$ with $\beta\in\bn$ and $\dset\subset \bn_0$ with $\gcd(\dset)=1$, note that for any $x_2>x_1\ge 1$,
\[
\left|S_{\betaD}(x_2)-S_{\betaD}(x_1)\right|\le \left(1+|x_2-x_1|\right)
\max_{x_1\le \ell\le x_2} r_{\betaD}(\ell).
\]
By Theorem \ref{th2}, choosing $n=-\lfloor\log_\beta |\Delta|\rfloor$ gives
\[
\left|S_{\betaD}(\beta^{n+x+\Delta})-S_{\betaD}(\beta^{n+x})\right|
\ll (1+\beta^n|\beta^{\Delta}-1|)\mu_{\betaD}^n\ll \mu_{\betaD}^n.
\]
Repeating the estimate in \eqref{eq1111} with Theorem \ref{th2} then yields
\begin{align*}
\Psi_{\betaD}(x+\Delta)-\Psi_{\betaD}(x)\ll |\Delta|+|\dset|^{(1-\log_{|\dset|}\mu_{\betaD})\log_\beta |\Delta|}
\ll|\Delta|^{\log_\beta |\dset|-\log_\beta \mu_{\betaD}},
\end{align*}
since $\log_\beta |\dset|-\log_\beta \mu_{\betaD}\ge \log_\beta |\dset|-\log_\beta(|\dset|/\beta)=1$. Thus $\Psi_{\betaD}\in \Lip_{\eta}(\bt)$ with $\eta=\log_\beta |\dset|-\log_\beta \mu_{\betaD}$.

We now show that $\Psi_{\betaD}\in \BV(\bt)$. For any $b\in\br$, $n\in\bn$, and any partition $b-1\le x_0<x_1<\cdots<x_k=b$ of $[b-1,b]$, the monotonicity of $S_{\betaD}$ implies
\begin{align*}
\sum_{0\le j<k}&\left|\frac{S_{\betaD}(\beta^{n+x_{j+1}})}{|\dset|^{n+x_{j+1}}}-\frac{S_{\betaD}(\beta^{n+x_j})}{|\dset|^{n+x_j}}\right|\\
&\qquad \le \sum_{0\le j<k}\left(\frac{S_{\betaD}(\beta^{n+x_{j+1}})-S_{\betaD}(\beta^{n+x_j})}{|\dset|^{n+x_{j+1}}}+\frac{S_{\betaD}(\beta^{n+x_j})}{|\dset|^n}\left(\frac{1}{|\dset|^{x_j}}-\frac{1}{|\dset|^{x_{j+1}}}\right)\right)\\
&\qquad \le \frac{S_{\betaD}(\beta^{n+b})-S_{\betaD}(\beta^{n+b-1})}{|\dset|^{n+b-1}}+\frac{S_{\betaD}(\beta^{n+b})}{|\dset|^n}\left(\frac{1}{|\dset|^{b-1}}-\frac{1}{|\dset|^{b}}\right)\le 2|\dset|\frac{S_{\betaD}(\beta^{n+b})}{|\dset|^{n+b}}.
\end{align*}
Taking the supremum over all partitions of $[b-1, b]$, we see that the total variation $V_{b-1}^b$ of 
$$g_n(x):=|\dset|^{-n-x}S_{\betaD}(\beta^{n+x})$$
 satisfies
\[
V_{b-1}^b(g_n)\le 2|\dset|\,g_n(b).
\]
Since $\lim_{n\rightarrow \infty}g_n(x)=\Psi_{\betaD}(x)$ pointwise on $[b-1,b]$, lower semicontinuity of total variation (see \cite[Chapter 5.2, Problem 9, p.~104]{MR1013117}) gives
\begin{align*}
V_{b-1}^b(\Psi_{\betaD})\le \liminf_{n\rrw \infty}V_{b-1}^b(g_n)\le 2|\dset|\liminf_{n\rrw \infty}g_n(b)=2|\dset|\Psi_{\betaD}(b)<+\infty.
\end{align*}
Thus $\Psi_{\betaD}\in\BV(\bt)$.

Finally, to prove that $\Psi_{\betaD}(x)$ is strictly positive, choose $x\ge \log_\beta\big(\dmax/(\beta-1)\big)$. For any $(\delta_0,\ldots,\delta_{n-1})\in\dset^n$, we have
\[
\sum_{0\le k<n}\delta_k\beta^k\le \dmax\frac{\beta^n-1}{\beta-1}<\frac{\dmax\beta^n}{\beta-1}\le \beta^{n+x},
\]
and hence
\begin{align*}
|\dset|^{-n-x}S_{\betaD}(\beta^{n+x})
&\ge |\dset|^{-n-x}\left|\left\{(\delta_0, \delta_1,\ldots,\delta_{n-1})\in \dset^{n}:
\sum_{0\le k<n}\delta_k\beta^k\le \beta^{n+x}
\right\}\right|=|\dset|^{-x}.
\end{align*}
Letting $n\rrw+\infty$ gives $\Psi_{\betaD}(x)\ge |\dset|^{-x}>0$ for all $x\ge \log_\beta\big(\dmax/(\beta-1)\big)$. By $1$-periodicity, $\Psi_{\betaD}(x)>0$ for all $x\in\br$.
\end{proof}

We now complete the proofs of \eqref{eqdas1} and \eqref{prosssc} in Theorem \ref{th1}.

\begin{proof}[Proof of \eqref{eqdas1}]
Set $m_0:=\lfloor\log_\beta(1/\dmin)\rfloor+2$. Since $\lim_{k\rightarrow +\infty} b_{k+1}/b_k = \beta$, for any fixed $x\in \br$ there exists $n_x\in\mathbb{Z}$ such that for all $n\geq n_x$, we have $\dmin b_{n+\lfloor x\rfloor+m_0}> \beta^x b_{n}$. Thus every representation $\sum_{k\ge 0}\delta_k b_k\le \beta^x b_n$ is supported on $0\le k< n+\lfloor x\rfloor+m_0$. Re-indexing $k\mapsto n-k$, we can write
\begin{align*}
S_{\bd}(\beta^xb_n)&=\left|\left\{(\delta_0,\delta_1,\ldots, \delta_{n+\lfloor x\rfloor+m_0-1})\in \dset^{n+\lfloor x\rfloor+m_0}:
\sum_{0\le k< n+\lfloor x\rfloor+m_0}\delta_k b_k\le \beta^{x}b_n\right\}\right|\\
&=\left|\left\{(\delta_{1-\lfloor x\rfloor-m_0},\ldots, \delta_{n})\in \dset^{n+\lfloor x\rfloor+m_0}:
\sum_{1-\lfloor x\rfloor-m_0\le k\le n}\delta_k \left(\beta^{-k}+ \frac{b_{n-k}}{b_n}-\beta^{-k}\right)\le \beta^{x}\right\}\right|.
\end{align*}
Setting
\begin{align*}
\varepsilon_{\betaD}(n,x):= \frac{\dmax}{\beta^x}\sum_{1-\lfloor x\rfloor-m_0\le k\le n}\left|\frac{b_{n-k}}{b_n}-\beta^{-k}\right|,
\end{align*}
we obtain the sandwich bounds
\[
S_{\betaD}\big((1-\varepsilon_{\betaD}(n,x))\beta^{n+x}\big)\le S_{\bd}(\beta^xb_n)\le S_{\betaD}\big((1+\varepsilon_{\betaD}(n,x))\beta^{n+x}\big).
\]
By the reverse Fatou lemma,
\begin{align*}
\limsup_{n\rrw +\infty}\varepsilon_{\betaD}(n,x)&\le \frac{\dmax}{\beta^x}\sum_{k\ge 1-\lfloor x\rfloor-m_0}\limsup_{n\rrw +\infty}\left(\left|\frac{b_{n-k}}{b_n}-\beta^{-k}\right|\ind{k\le n}\right)=0.
\end{align*}
Combining Lemma \ref{lemm21} with Proposition \ref{lemm20} ($\Psi_{\betaD}\in \Lip_{\eta}(\bt)$ for $\eta=\kappa\log_\beta |\dset|$ with $\kappa<\kappa_{\betaD}$), we find that as $n\rrw +\infty$,
\begin{align*}
\frac{S_{\betaD}\left((1\pm \varepsilon_{\betaD}(n,x))\beta^{n+x}\right)}{|\dset|^{n+x}}&=|\dset|^{\log_\beta(1\pm \varepsilon_{\betaD}(n,x))}\Psi_{\betaD}\left(x+\log_\beta (1\pm \varepsilon_{\betaD}(n,x))\right)+O_x\left(|\dset|^{-\kappa n}\right)\\
&=\Psi_{\betaD}(x)+O_x\left(\varepsilon_{\betaD}(n,x)+\varepsilon_{\betaD}(n,x)^{\kappa \log_\beta |\dset|}+|\dset|^{-\kappa n}\right).
\end{align*}
Therefore,
\begin{equation}\label{eqgsas}
|\dset|^{-n-x}S_{\bd}(\beta^xb_n)=\Psi_{\betaD}(x)+O_x\left(\varepsilon_{\betaD}(n,x)^{\kappa \log_\beta |\dset|}+|\dset|^{-\kappa n}\right).
\end{equation}
Since $\lim_{n\rrw +\infty}\varepsilon_{\betaD}(n,x)=0$, the scaling limit \eqref{eqdas1} follows.
\end{proof}

\begin{proof}[Proof of \eqref{prosssc}]
Write $x=\beta^{n+u}$ with $n=\lfloor \log_\beta x\rfloor$ and $u=\{\log_\beta x\}\in[0,1)$, and let 
$$m_0:=\lfloor\log_\beta(1/\dmin)\rfloor+2.$$ For all sufficiently large $x$,
\begin{align*}
S_{\bd}(x)&=\left|\left\{(\delta_0,\delta_1,\ldots ,\delta_{n+m_0})\in \dset^{n+m_0+1}:\sum_{0\le k\le n+m_0}\delta_k (\beta^k+b_k-\beta^k)\le \beta^{n+u}\right\}\right|.
\end{align*}
Defining
\begin{align*}
\varepsilon_{\betaD}(x):= \dmax\beta^{-n}\sum_{0\le k\le n+m_0}\left|b_k-\beta^k\right|,
\end{align*}
we have
\[
S_{\betaD}\big((1-\varepsilon_{\betaD}(x))x\big)\le S_{\bd}(x)\le S_{\betaD}\big((1+\varepsilon_{\betaD}(x))x\big).
\]
Under the hypothesis $b_k=\beta^k+O(\beta^{(1-\gamma)k})$ with $\gamma\in(0,1]$, we estimate
\begin{align*}
\varepsilon_{\betaD}(x)\ll \beta^{-n}\sum_{0\le k\le n+m_0}\beta^{(1-\gamma)k}
\ll_\gamma \beta^{-n}n\ind{\gamma=1}+\beta^{-\gamma n}\ll (x^{-1}\log x)\ind{\gamma=1}+x^{-\gamma}.
\end{align*}
Arguing as in \eqref{eqgsas}, we obtain for any $\kappa<\kappa_{\betaD}$,
\begin{align*}
x^{-\log_\beta |\dset|}S_{\bd}(x)&=\Psi_{\betaD}(\log_\beta x)+O\left(\varepsilon_{\betaD}(x)^{\kappa \log_\beta |\dset|}+|\dset|^{-\kappa \log_\beta x}\right)\\
&=\Psi_{\betaD}(\log_\beta x)+O\left(x^{-\kappa \log_\beta |\dset|}(\log x)^{\kappa \log_\beta |\dset|}+x^{-\kappa \gamma\log_\beta |\dset|}\right),
\end{align*}
which proves \eqref{prosssc}.
\end{proof}

\subsection{Moments of the representation function \texorpdfstring{$r_{\bd}(\lambda)$}{r(lambda)}}

We now investigate the arbitrary moments
\[
\sum_{0< \lambda\le x}\lambda^kr_{\bd}(\lambda) \qquad (k\in \br)
\]
as $x\rrw+\infty$. For fixed $k, x\in\br$ and $n\in\bn$, Abel summation gives
\begin{align*}
\sum_{\beta^xb_0< \lambda\le \beta^xb_n}\lambda^kr_{\bd}(\lambda)= (\beta^xb_n)^kS_{\bd}(\beta^xb_n)-(\beta^xb_0)^kS_{\bd}(\beta^xb_0)-k\int_{\beta^xb_0}^{\beta^xb_n}S_{\bd}(u)u^{k-1}\,\rd u.
\end{align*}
Invoking the scaling limit \eqref{eqdas1} of Theorem \ref{th1}, we have as $n\rrw\infty$,
\begin{align}\label{eqprim}
\sum_{0< \lambda\le \beta^xb_n}\lambda^kr_{\bd}(\lambda)= (\beta^xb_n)^k|\dset|^{n+x}\Psi_{\betaD}(x)(1+o(1))-k\int_{\beta^xb_0}^{\beta^xb_n}S_{\bd}(u)u^{k-1}\,\rd u+O(1).
\end{align}
Splitting the integral over the intervals $[\beta^x b_h, \beta^x b_{h+1}]$ ($0\le h<n$) and changing variables gives
\begin{align*}
\int_{\beta^xb_0}^{\beta^xb_n}S_{\bd}(u)u^{k-1}\,\rd u=\sum_{0\le h< n}|\dset|^{h+x}(\beta^xb_h)^k\int_{1}^{\frac{b_{h+1}}{b_{h}}}\frac{S_{\bd}(\beta^{x+\log_\beta u}b_{h})}{|\dset|^{h+x+\log_\beta u}}u^{k+\log_\beta|\dset|-1}\,\rd u.
\end{align*}
Since $\lim_{h\rrw\infty}b_{h+1}/b_h=\beta$, \eqref{eqdas1} and the dominated convergence theorem yield
\[
\lim_{h\rrw\infty}\int_{1}^{\frac{b_{h+1}}{b_{h}}}\frac{S_{\bd}(\beta^{x+\log_\beta u}b_{h})}{|\dset|^{h+x+\log_\beta u}}u^{k+\log_\beta|\dset|-1}\,\rd u=
\int_{1}^{\beta}\Psi_{\betaD}(x+\log_\beta u)u^{k+\log_\beta|\dset|-1}\,\rd u.
\]
We may therefore write
\begin{align*}
\int_{\beta^xb_0}^{\beta^xb_n}S_{\bd}(u)u^{k-1}\,\rd u=|\dset|^{x}\beta^{kx}\sum_{0\le h< n}|\dset|^{h}b_h^k\left(I_{\betaD}^{(k)}(x)+E_{\bd}^{(k)}(x,h)\right),
\end{align*}
where the main term
\begin{align}\label{eqiii}
I_{\betaD}^{(k)}(x)&=\int_{1}^{\beta}\Psi_{\betaD}(x+\log_\beta u)u^{k+\log_\beta|\dset|-1}\,\rd u\nonumber\\
&=(\log \beta)\int_{0}^{1}\Psi_{\betaD}(x+w)\beta^{(k+\log_\beta|\dset|)w}\,\rd w
\end{align}
is strictly positive by Proposition \ref{lemm20}, and the error term
\[
E_{\bd}^{(k)}(x,h)=\int_{1}^{\frac{b_{h+1}}{b_{h}}}\frac{S_{\bd}(\beta^{x+\log_\beta u}b_{h})}{|\dset|^{h+x+\log_\beta u}}u^{k+\log_\beta|\dset|-1}\,\rd u-I_{\betaD}^{(k)}(x)
\]
satisfies $\lim_{h\rrw\infty}E_{\bd}^{(k)}(x,h)=0$. Substituting this into \eqref{eqprim} gives, as $n\rrw\infty$,
\begin{align}\label{eq222}
\sum_{0< \lambda\le \beta^xb_n}\lambda^kr_{\bd}(\lambda)=&\,(|\dset|^{n}b_n^k)|\dset|^x\beta^{kx}\Psi_{\betaD}(x)(1+o(1))\nonumber\\
&-k|\dset|^{x}\beta^{kx}I_{\betaD}^{(k)}(x)\sum_{0\le h< n}\left(1+\widetilde{E}_{\bd}^{(k)}(x,h)\right)|\dset|^{h}b_h^k+O(1),
\end{align}
where $\widetilde{E}_{\bd}^{(k)}(x,h)=E_{\bd}^{(k)}(x,h)/I_{\betaD}^{(k)}(x)$.

\medskip

We can now complete the proof of Corollary \ref{cor1}. For $k>-\log_\beta |\dset|$, Lebesgue's dominated convergence theorem yields
\begin{align*}
\lim_{n\rrw +\infty}\sum_{0\le h< n}\left(1+\widetilde{E}_{\bd}^{(k)}(x,h)\right)\frac{|\dset|^{h}b_h^k}{|\dset|^{n}b_n^k}
&=\sum_{h\ge 1}\lim_{n\rrw +\infty}\left(1+\widetilde{E}_{\bd}^{(k)}(x,n-h)\right)\frac{|\dset|^{n-h}b_{n-h}^k}{|\dset|^{n}b_n^k}\ind{h\le n}\\
&=\sum_{h\ge 1}(|\dset|\beta^{k})^{-h}=\frac{1}{|\dset|\beta^{k}-1}.
\end{align*}
Hence, as $n\rrw\infty$,
\begin{align*}
\frac{(\beta^xb_n)^{-k}}{|\dset|^{n+x}}\sum_{0< \lambda\le \beta^xb_n}\lambda^kr_{\bd}(\lambda)=\Psi_{\betaD}(x)-\frac{k}{|\dset|\beta^{k}-1}I_{\betaD}^{(k)}(x)+o(1).
\end{align*}
Substituting \eqref{eqiii} and replacing $k$ by $k-\log_\beta |\dset|$ gives \eqref{eqrekm} in Corollary \ref{cor1}.

\smallskip

For $k<-\log_\beta|\dset|$, since $\log_\beta b_h = h(1+o(1))$, the series $\sum_{h\ge 0}|\dset|^h b_h^k$ converges, so \eqref{eq222} shows that $\sum_{\lambda>0}\lambda^k r_{\bd}(\lambda)<+\infty$. At the critical exponent $k=-\log_\beta |\dset|$, if $\sum_{h\ge 0}|\dset|^{h-\log_\beta b_h}<+\infty$, then \eqref{eq222} likewise gives
\begin{align*}
\sum_{\lambda>0}\frac{r_{\bd}(\lambda)}{\lambda^{\log_\beta |\dset|}}
\ll 1+\lim_{n\rrw\infty}\sum_{0\le h<n}|\dset|^{h}b_h^{-\log_\beta |\dset|}=1+\sum_{h\ge 0}|\dset|^{h-\log_\beta b_h}<+\infty.
\end{align*}
If instead $\sum_{h\ge 0}|\dset|^{h-\log_\beta b_h}=+\infty$, then substituting $\lim_{h\rrw+\infty}\widetilde{E}_{\bd}^{(k)}(x,h)=0$ and \eqref{eqiii} into \eqref{eq222} gives
\begin{align*}
\sum_{0<\lambda\le \beta^xb_n}\frac{r_{\bd}(\lambda)}{\lambda^{\log_\beta |\dset|}}\sim |\dset|^{n-\log_\beta b_n}\Psi_{\betaD}(x)+(\log_\beta |\dset|)\sum_{0\le h< n}|\dset|^{h-\log_\beta b_h}\int_{0}^1\Psi_{\betaD}(u)\,\rd u
\end{align*}
as $n\rrw+\infty$. Applying Lemma \ref{lemaa} below to $a_h=|\dset|^{h-\log_\beta b_h}$ shows that the first term on the right-hand side is $o\big(\sum_{0\le h<n}a_h\big)$, which establishes \eqref{eqrelog} and completes the proof of Corollary \ref{cor1}.

\begin{lemma}\label{lemaa}
Let $(a_k)_{k\ge 0}$ be a sequence of positive real numbers such that $\lim_{n\rrw\infty}a_{n+1}/a_n=1$. Then
\[
\lim_{n\rrw \infty}\frac{a_n}{\sum_{0\le k< n}a_k}=0.
\]
\end{lemma}

\begin{proof}
Since $\lim_{n\rrw\infty}a_{n+1}/a_n=1$, for any $\varepsilon\in(0,1)$ there exists $n_\varepsilon\in\bn$ such that $a_{k}\ge (1-\varepsilon)a_{k+1}$ for all $k\ge n_{\varepsilon}$, and hence $a_k\ge (1-\varepsilon)^{n-k}a_n$ for all $n>k\ge n_\varepsilon$. Consequently,
\[
\frac{a_n}{\sum_{0\le k< n}a_k}\le \frac{a_n}{\sum_{n_{\varepsilon}\le k< n}(1-\varepsilon)^{n-k}a_n}=\frac{\varepsilon}{(1-\varepsilon)\big(1-(1-\varepsilon)^{n-n_\varepsilon}\big)}.
\]
Letting $n\rrw\infty$ and then $\varepsilon\rrw 0^+$ yields the lemma.
\end{proof}

\section{Generating functions and associated zeta functions of \texorpdfstring{$r_{\bd}(\lambda)$}{r(lambda)}}\label{sec3}

In this section, we study the asymptotic behavior of the generating function of $r_{\bd}(\lambda)$ as $t\rrw 0^+$ and establish the meromorphic continuation of $\zeta_{\bd}(s)$, proving Proposition \ref{pro36} and Theorem \ref{th3}.

\subsection{Generating functions of \texorpdfstring{$r_{\bd}(\lambda)$}{r(lambda)}}\label{sec31}

From \eqref{eqr}, the Laplace-type generating function of $r_{\bd}(\lambda)$ factors as an infinite product:
\begin{equation}\label{eqm1}
Z_{\bd}(e^{-t}):=\sum_{\lambda \in \cN_{\bd}}r_{\bd}(\lambda)e^{-\lambda t}=\prod_{k\ge 0}\sum_{ \delta\in \dset}e^{-t\delta b_k},
\end{equation}
which converges absolutely for $\re(t)>0$. For $y\ge 0$, define
\begin{equation}\label{eq1}
L_{\dset}(y):=\log\left(\sum_{\delta\in \dset}e^{-\delta y}\right),
\end{equation}
and for $w\in\br$, define
\begin{equation}\label{eq2}
P_{\betaD}(w):=\sum_{k\in\bz}\left(k+w+\frac{1}{2}\right)\left(L_{\dset}(\beta^{k+w})-L_{\dset}(\beta^{k+1+w})\right)+\frac{1}{2}\log|\dset|.
\end{equation}
One readily verifies that the series in \eqref{eq2} converges absolutely and defines a smooth $1$-periodic function on $\br$. We first establish a fundamental decomposition for $\sum_{k\ge 0}L_{\dset}(\beta^k t)$.

\begin{proposition}\label{prop1}
For any $\beta>1$ and all $t>0$, we have
\begin{align*}
\sum_{k\ge 0}L_{\dset}(\beta^{k}t)=-(\log_\beta t)\log |\dset|+P_{\betaD}(\log_\beta t)+\sum_{k\ge 1} \left(\log|\dset|-L_{\dset}(\beta^{-k}t)\right).
\end{align*}
\end{proposition}

\begin{proof}
For $a,b\in\bz$ and $f\in C^1([a, b])$, the Euler--Maclaurin summation formula states that
\[
\sum_{a\le\ell\le b}f(\ell)=\int_{a}^b\left(f(x)+\psi(x)f'(x)\right)\,\rd x+\frac{1}{2}\big(f(b)+f(a)\big),
\]
where $\psi(x)=x-\lfloor x\rfloor-1/2$. Applying this to $f(x)=L_{\dset}(\beta^{x+w})$ on $[0,b]$ and letting $b\rrw+\infty$, we obtain via integration by parts
\begin{align*}
\sum_{\ell\ge 0}L_{\dset}(\beta^{\ell+w})
&=\int_{0}^\infty\left(L_{\dset}(\beta^{x+w})+\psi(x)L_{\dset}'(\beta^{x+w})\beta^{x+w}\log\beta\right)\,\rd x+\frac{1}{2}L_{\dset}(\beta^w)\\
&=\int_{w}^\infty\left(\psi(x-w)-x\right)L_{\dset}'(\beta^{x})\beta^{x}\log\beta\,\rd x+\left(\frac{1}{2}-w\right)L_{\dset}(\beta^w).
\end{align*}
Substituting $x=w+\log_\beta y$ yields
\begin{align*}
\sum_{\ell\ge 0}L_{\dset}(\beta^{\ell+w})
&=\int_{1}^\infty\left(\psi(\log_\beta y)-\log_\beta y-w\right)L_{\dset}'(\beta^{w}y)\beta^{w}\,\rd y+\left(\frac{1}{2}-w\right)L_{\dset}(\beta^w)\\
&=-\int_{1}^\infty\left(\lfloor\log_\beta y\rfloor+w+\frac{1}{2}\right)L_{\dset}'(\beta^{w}y)\beta^{w}\,\rd y+\left(\frac{1}{2}-w\right)L_{\dset}(\beta^w).
\end{align*}
By \eqref{eq2}, the full integral over $(0,\infty)$ evaluates to
\begin{align*}
\int_{0}^\infty\left(\lfloor\log_\beta y\rfloor+w+\frac{1}{2}\right)L_{\dset}'(\beta^{w}y)\beta^{w}\,\rd y
&=\sum_{k\in\bz}\left(k+w+\frac{1}{2}\right)\int_{\beta^{k}}^{\beta^{k+1}}L_{\dset}'(\beta^{w}y)\,\rd (\beta^{w}y)\\
&=\sum_{k\in\bz}\left(k+w+\frac{1}{2}\right)\left(L_{\dset}(\beta^{k+1+w})-L_{\dset}(\beta^{k+w})\right)\\
&=\frac{1}{2}\log|\dset|-P_{\betaD}(w).
\end{align*}
Consequently,
\begin{align*}
\sum_{\ell\ge 0}L_{\dset}(\beta^{\ell+w})
&=P_{\betaD}(w)-\frac{1}{2}\log|\dset|+\int_{0}^1\left(\lfloor\log_\beta y\rfloor+w+\frac{1}{2}\right)L_{\dset}'(\beta^{w}y)\beta^{w}\,\rd y+\left(\frac{1}{2}-w\right)L_{\dset}(\beta^w)\\
&=P_{\betaD}(w)+\int_{0}^1\left(\lfloor\log_\beta y\rfloor+1\right)L_{\dset}'(\beta^{w}y)\,\rd (\beta^{w}y)-w L_{\dset}(0).
\end{align*}
Finally, integrating over $(\beta^k,\beta^{k+1})$ for $k\le -1$ and summing by parts gives
\begin{align*}
\int_{0}^1\left(\lfloor\log_\beta y\rfloor+1\right)L_{\dset}'(\beta^{w}y)\,\rd (\beta^{w}y)
&=\sum_{k=-\infty }^{-1}(k+1)\left(\big(L_{\dset}(\beta^{k+1+w})-\log|\dset|\big)-\big(L_{\dset}(\beta^{k+w})-\log|\dset|\big)\right)\\
&=-\sum_{k=-\infty }^{-1}\left(L_{\dset}(\beta^{k+w})-\log|\dset|\right).
\end{align*}
Setting $w=\log_\beta t$ and noting that $L_{\dset}(0)=\log|\dset|$ completes the proof.
\end{proof}

\begin{remark}
One can also verify Proposition \ref{prop1} directly by telescoping partial sums: for any $w\in\br$,
\begin{align*}
&P_{\betaD}(w)-\frac{1}{2}\log|\dset|\\
=&\lim_{N\rrw +\infty}\sum_{-N\le k<N}\left(k+w+\frac{1}{2}\right)\left(L_{\dset}(\beta^{k+w})-L_{\dset}(\beta^{k+1+w})\right)\\
=&\lim_{N\rrw +\infty}\left(\left(\frac{1}{2}+w-N\right)L_{\dset}(\beta^{w-N})-\left(N+w+\frac{1}{2}\right)L_{\dset}(\beta^{N+w})+\sum_{-N\le k<N}L_{\dset}(\beta^{k+1+w})\right).
\end{align*}
Since $L_{\dset}(\beta^{w-N})=\log|\dset|+O(\beta^{w-N})$ and $L_{\dset}(\beta^{N+w})=O(e^{-\dmin\beta^{N+w}})$ as $N\rrw+\infty$, taking the limit and setting $w=\log_\beta t$ recovers Proposition \ref{prop1}.
\end{remark}

Exponentiating the identity in Proposition \ref{prop1} for the geometric base $\bseq=(\beta^k)_{k\ge 0}$ gives
\begin{align*}
Z_{\betaD}(e^{-t})=t^{-\log_\beta |\dset|}e^{P_{\betaD}(\log_\beta t)}\prod_{k\ge 1}\frac{|\dset|}{\sum_{\delta\in \dset}e^{-\delta\beta^{-k}t}}.
\end{align*}
Since the infinite product on the right-hand side converges absolutely and uniformly in a neighborhood of $t=0$, it defines a holomorphic function on $|t|<R_{\betaD}$ for some radius $R_{\betaD}>0$, admitting a Maclaurin expansion
\begin{equation*}
\sum_{\ell\ge 0}c_{\betaD}(\ell)t^\ell:=\prod_{k\ge 1}\frac{|\dset|}{\sum_{\delta\in \dset}e^{-\delta\beta^{-k}t}} \qquad (|t|<R_{\betaD}).
\end{equation*}
We set
\begin{equation}\label{eqpr1}
\rho_{\betaD}=\min\{h\in\bn_0: \beta^{-h}<R_{\betaD}\}.
\end{equation}

\begin{remark}\label{rem32}
The quantities $\ell !\,c_{\betaD}(\ell)$ may be viewed as generalized Bernoulli numbers. Indeed, when $\beta\in\bn$ and $\dset=\{0,1,\ldots,\beta-1\}$, telescoping the product gives
\[
\prod_{k\ge 1}\frac{|\dset|}{\sum_{\delta\in \dset}e^{-\delta\beta^{-k}t}}=\prod_{k\ge 1}\frac{\beta(1-e^{-\beta^{-k}t})}{1-e^{-\beta^{1-k}t}}=\frac{1}{1-e^{-t}}\lim_{N\rrw+\infty}\beta^N(1-e^{-\beta^{-N}t})=\frac{t}{1-e^{-t}},
\]
so that $\ell !\,c_{\betaD}(\ell)=(-1)^\ell B_\ell$. Similarly, $t Z_{\betaD}(e^{-t})=t/(1-e^{-t})$, which implies $P_{\betaD}(w)\equiv 0$ in the classical $\beta$-adic system.
\end{remark}

For a general digit set $\dset$, let $[y^h]L_{\dset}(y)$ denote the coefficient of $y^h$ in the Maclaurin expansion of $L_{\dset}(y)$. Expanding $\log\big(\sum_{\ell\ge 0}c_{\betaD}(\ell)t^\ell\big)$ for small $|t|$ yields
\begin{align*}
\sum_{\ell\ge 0}c_{\betaD}(\ell)t^\ell&=\exp\left(\sum_{k\ge 1} \left(\log|\dset|-L_{\dset}(\beta^{-k}t)\right)\right)\\
&=\exp\left(-\sum_{h\ge 1}t^h([y^h]L_{\dset}(y))\sum_{k\ge 1}\beta^{-hk}\right)=\prod_{h\ge 1}\exp\left(t^h\frac{[y^h]L_{\dset}(y)}{1-\beta^h}\right).
\end{align*}
Expanding each exponential factor gives the explicit formula
\begin{equation}\label{eq2'}
c_{\betaD}(m)=\sum_{\substack{\ell_1,\ell_2,\ldots,\ell_m\ge 0\\ 1\cdot \ell_1+2\cdot \ell_2+\cdots+m\ell_m=m}}\prod_{1\le h\le m}\frac{1}{\ell_h!}\left(\frac{[y^h]L_{\dset}(y)}{1-\beta^h}\right)^{\ell_h} \qquad (m\in\bn_0).
\end{equation}
We summarize this discussion in the following lemma.

\begin{lemma}\label{lem32}
For all $0<t\le \beta^{-\rho_{\betaD}}$, we have
\[
Z_{\betaD}\left(e^{-t}\right)=t^{-\log_\beta |\dset|}e^{P_{\betaD}\left(\log_\beta t\right)}\sum_{m\ge 0}c_{\betaD}(m)t^m,
\]
where the coefficients $c_{\betaD}(m)$ are given by \eqref{eq2'}.
\end{lemma}

For perturbed geometric bases, we have the following asymptotic expansion.

\begin{lemma}\label{lem34}
If $b_k=c\beta^{k}+O(\beta^{(1-\gamma)k})$ for constants $c>0$ and $\gamma>0$, then for $t\in(0,1]$,
\[
Z_{\bd}\left(e^{-t/c}\right)=e^{P_{\betaD}\left(\log_\beta t\right)}t^{-\log_\beta |\dset|}+t^{-\log_\beta |\dset|+\min(1,\gamma)}\left(\log\frac{2}{t}\right)^{\ind{\gamma=1}}B_{\bd}(t),
\]
where $B_{\bd}(t)$ is bounded on $(0,1]$.
\end{lemma}

\begin{proof}
Write $b_k=c\beta^{k}+c\theta_k\beta^{(1-\gamma)k}$ for a bounded real sequence $(\theta_k)_{k\ge 0}$. By the fundamental theorem of calculus,
\begin{align*}
\log Z_{\bd}\left(e^{-t/c}\right)&=\sum_{k\ge 0}L_{\dset}\big(t(\beta^{k}+\theta_k\beta^{(1-\gamma)k})\big)\\
&=\sum_{k\ge 0}L_{\dset}(t\beta^{k})+\sum_{k\ge 0}\int_{0}^{1}t\theta_k \beta^{(1-\gamma)k}L_{\dset}'\big(\beta^{k}t+vt\theta_k \beta^{(1-\gamma)k}\big)\,\rd v.
\end{align*}
The first sum is governed by Lemma \ref{lem32}. For the remainder sum, note that for all $y\ge 0$,
\[
|L_{\dset}'(y)|=\frac{\sum_{\delta\in \dset}\delta e^{-\delta y}}{\sum_{\delta\in \dset}e^{-\delta y}}\le e^{-\dmin y} (|\dset|-1)\dmax.
\]
Hence, for $t\in(0,1]$, splitting the sum at $k=\lfloor\log_\beta(1/t)\rfloor$ gives
\begin{align*}
\sum_{k\ge 0}\int_{0}^{1}t|\theta_k| \beta^{(1-\gamma)k}\left|L_{\dset}'\big(t\beta^{k}+vt\theta_k\beta^{(1-\gamma)k}\big)\right|\,\rd v
&\ll t^\gamma\sum_{k\ge 0}(t \beta^{k})^{1-\gamma}\exp\left(-\frac{1}{2}\dmin t\beta^{k}\right)\\
&\ll t^{\min(1,\gamma)}\left(\log\frac{2}{t}\right)^{\ind{\gamma=1}}.
\end{align*}
Exponentiating both sides and combining with Lemma \ref{lem32} completes the proof.
\end{proof}

\begin{proof}[Proof of Proposition \ref{pro36}]
By Proposition \ref{lemm20} and Zygmund's theorem \cite[Theorem (3.6), p.~241]{MR1963498}, $\Psi_{\betaD}(x)$ admits an absolutely convergent Fourier series $\Psi_{\betaD}(x)=\sum_{k\in\bz}\widehat{\Psi}_{\betaD}(k)e^{2\pi\ri kx}$. To evaluate $\widehat{\Psi}_{\betaD}(k)=\int_{0}^1\Psi_{\betaD}(x)e^{-2\pi\ri kx}\,\rd x$, we relate $\Psi_{\betaD}$ to $Z_{\betaD}(e^{-t})$ via Laplace--Stieltjes integration by parts:
\begin{align*}
Z_{\betaD}\left(e^{-t}\right)=\sum_{\lambda\in\cN_{\betaD}}r_{\betaD}(\lambda)e^{-\lambda t}=t\int_{0}^{\infty}S_{\betaD}(u) e^{-tu}\,\rd u.
\end{align*}
Splitting the integral at $u=1$ and changing variables $tu\mapsto u$ gives
\begin{align*}
Z_{\betaD}\left(e^{-t}\right)-(1-e^{-t})&=\int_{t}^{\infty}S_{\betaD}(u/t) e^{-u}\,\rd u\\
&=\int_{t}^{\infty}(u/t)^{\log_\beta |\dset|}\left(\frac{S_{\betaD}(u/t)}{(u/t)^{\log_\beta |\dset|}}-\Psi_{\betaD}\left(\log_\beta(u/t)\right)\right) e^{-u}\,\rd u\\
&\quad+\int_{t}^{\infty}\Psi_{\betaD}\left(\log_\beta(u/t)\right)(u/t)^{\log_\beta |\dset|} e^{-u}\,\rd u.
\end{align*}
Since $x^{-\log_\beta|\dset|}S_{\betaD}(x)$ and $\Psi_{\betaD}(\log_\beta x)$ are bounded on $x\ge 1$ and 
$$\lim_{x\rrw+\infty}\big(x^{-\log_\beta |\dset|}S_{\betaD}(x)-\Psi_{\betaD}(\log_\beta x)\big)=0$$ 
by \eqref{eqdas1}, the integrand in the first integral is dominated by $C u^{\log_\beta|\dset|}e^{-u}\in L^1(0,\infty)$. Lebesgue's dominated convergence theorem therefore implies that the first integral is $o(t^{-\log_\beta|\dset|})$ as $t\rrw 0^+$. Thus, as $t\rrw 0^+$,
\begin{align*}
t^{\log_\beta |\dset|}Z_{\betaD}\left(e^{-t}\right)=\int_{t}^{\infty}\Psi_{\betaD}\left(\log_\beta(u/t)\right)u^{\log_\beta |\dset|} e^{-u}\,\rd u+o(1).
\end{align*}
Using the $1$-periodicity of $\Psi_{\betaD}$ and substituting $u=\beta^v t$, we find
\begin{align*}
\int_{t}^{\infty}\Psi_{\betaD}\left(\log_\beta(u/t)\right)u^{\log_\beta |\dset|} e^{-u}\,\rd u&=(\log\beta)\int_{0}^\infty\Psi_{\betaD}(v)(\beta^{v}t)^{1+\log_\beta |\dset|}e^{-\beta^{v}t}\,\rd v\\
&=(\log\beta)\int_{0}^1\Psi_{\betaD}(u)\sum_{k\ge 0}(\beta^{k+u}t)^{1+\log_\beta |\dset|}e^{-\beta^{k+u}t}\,\rd u.
\end{align*}
By Mellin inversion for $e^{-y}$, for any $c>0$,
\begin{align*}
\sum_{k\ge 0}(\beta^{k+u}t)^{1+\log_\beta |\dset|}e^{-\beta^{k+u}t}&=\frac{1}{2\pi\ri}\int_{(c)}\frac{\Gamma(s+1+\log_\beta |\dset|)(\beta^ut)^{-s}}{1-\beta^{-s}}\,\rd s.
\end{align*}
Shifting the line of integration to $\re(s)=-1$ (justified by Stirling's formula), we encounter simple poles at $s=2\pi(\log \beta)^{-1}k\ri$ ($k\in\bz$). Summing the residues yields
\begin{align*}
\sum_{k\ge 0}(\beta^{k+u}t)^{1+\log_\beta |\dset|}e^{-\beta^{k+u}t}=&\sum_{k\in\bz}\frac{\Gamma\left(1+\log_\beta |\dset|+2\pi(\log \beta)^{-1}k\ri\right)e^{-2\pi\ri k(u+\log_\beta t)}}{\log\beta}\\
&+\frac{1}{2\pi\ri}\int_{(-1)}\frac{\Gamma(s+1+\log_\beta |\dset|)(\beta^ut)^{-s}}{1-\beta^{-s}}\,\rd s.
\end{align*}
Comparing this with Lemma \ref{lem32}, which gives $t^{\log_\beta |\dset|}Z_{\betaD}(e^{-t})=\exp(P_{\betaD}(\log_\beta t))(1+O(t))$, and setting $t=\beta^{-N+w}$ with $N\in\bn$, $N\rrw+\infty$, we deduce that
\begin{align*}
\exp\left(P_{\betaD}(w)\right)=\sum_{k\in\bz}e^{-2\pi\ri k w}\Gamma\left(1+\log_\beta |\dset|+2\pi(\log \beta)^{-1}k\ri\right)\int_{0}^1\Psi_{\betaD}(u)e^{-2\pi\ri ku}\,\rd u.
\end{align*}
Multiplying by $e^{2\pi\ri kw}$ and integrating over $w\in[0,1]$ yields
\[
\Gamma\left(1+\log_\beta |\dset|+2\pi(\log \beta)^{-1}k\ri\right)\widehat{\Psi}_{\betaD}(k)=\int_{0}^1e^{P_{\betaD}(w)}e^{2\pi\ri kw}\,\rd w.
\]
Substituting \eqref{eq2} and shifting $w\mapsto w-1/2$ completes the proof of Proposition \ref{pro36}.
\end{proof}

\subsection{Analytic continuation of \texorpdfstring{$\zeta_{\bd}(s)$}{zeta(s)}}\label{sec32}

Finally, we establish Theorems \ref{th34} and \ref{th35} below, which together with Proposition \ref{pro36} immediately imply Theorem \ref{th3}. Recall the upper incomplete gamma function
\[
\Gamma(s,w):=\int_{w}^{\infty}e^{-u}u^{s-1}\,\rd u,
\]
which is an entire function of $s\in\bc$ whenever $w>0$, and satisfies $\Gamma(s,0)=\Gamma(s)$.

\begin{theorem}\label{th34}
Let $\bseq=(\beta^k)_{k\ge 0}$ and let $c\ge \rho_{\betaD}$ be an integer, where $\rho_{\betaD}$ is defined by \eqref{eqpr1}. Then $\zeta_{\betaD}(s)$ admits a meromorphic continuation to the entire complex plane $\bc$, given by
\begin{align}\label{eqzebme}
\zeta_{\betaD}(s)=
&\,\frac{1}{\Gamma(s)}\sum_{\lambda>0}\frac{r_{\betaD}(\lambda)}{\lambda^s}\Gamma\left(s,\lambda\beta^{-c}\right)-\frac{\beta^{-cs}}{\Gamma(s+1)}\nonumber\\
&+\frac{1}{\Gamma(s)}\sum_{\ell\ge 0}\frac{\beta^{-(s+\ell-\log_\beta |\dset|)c}\log\beta}{\beta^{s+\ell-\log_\beta |\dset|}-1}
c_{\betaD}(\ell)\int_{0}^{1}e^{P_{\betaD}\left(u\right)}\beta^{(s+\ell-\log_\beta |\dset|) u}\,\rd u.
\end{align}
Thus $\zeta_{\betaD}(s)$ is holomorphic on $\bc$ except for possible simple poles at
\[
p_{\betaD}(j,k)=\log_\beta |\dset|-j-\frac{2\pi\ri k}{\log \beta} \qquad (j\in\bn_0,\, k\in\bz),
\]
with residues
\begin{align*}
\res_{s=p_{\betaD}(j,k)}\zeta_{\betaD}(s)=\frac{c_{\betaD}(j)}{\Gamma\big(\log_\beta |\dset|-j-2\pi(\log \beta)^{-1}k\ri\big)}\int_{0}^{1}e^{P_{\betaD}\left(u\right)}e^{-2\pi\ri k u}\,\rd u,
\end{align*}
where it is understood that $1/\Gamma(-n)=0$ for $n\in\bn_0$. Moreover, if $\log_\beta |\dset|\notin\bn$, then $\zeta_{\betaD}(0)=-1$ and $\zeta_{\betaD}(-n)=0$ for all $n\in\bn$; if $\log_\beta |\dset|\in\bn$, then for each $n\in\bn_0$,
\begin{align*}
\zeta_{\betaD}(-n)=(-1)^nn!\,c_{\betaD}(n+\log_\beta |\dset|)\int_{0}^1e^{P_{\betaD}\left(u\right)}\,\rd u-\ind{n=0}.
\end{align*}
\end{theorem}

\begin{proof}
For $\re(s)>\log_\beta|\dset|$, Mellin inversion gives
\begin{align*}
\Gamma(s)\zeta_{\betaD}(s)=\int_{0}^{\infty}\left(Z_{\betaD}\left(e^{-t}\right)-1\right)t^{s-1}\,\rd t.
\end{align*}
Splitting the integral at $t=\beta^{-c}$ ($c\ge \rho_{\betaD}$) yields
\begin{align}\label{eq000}
\zeta_{\betaD}(s)=\frac{1}{\Gamma(s)}\int_{0}^{\beta^{-c}}Z_{\betaD}\left(e^{-t}\right)t^{s-1}\,\rd t-\frac{\beta^{-cs}}{s\Gamma(s)}
+\frac{1}{\Gamma(s)}\sum_{\lambda>0}\frac{r_{\betaD}(\lambda)}{\lambda^s}\Gamma(s, \lambda \beta^{-c}).
\end{align}
By Lemma \ref{lem32}, for $\re(s)>\log_\beta |\dset|$, we may expand $Z_{\betaD}(e^{-t})$ on $(0,\beta^{-c}]$ and integrate term-by-term:
\begin{align}\label{eqppd}
\int_{0}^{\beta^{-c}}Z_{\betaD}\left(e^{-t}\right)t^{s-1}\,\rd t
&=\sum_{\ell\ge 0}c_{\betaD}(\ell)\int_{0}^{\beta^{-c}}e^{P_{\betaD}\left(\log_\beta t\right)}t^{s+\ell-1-\log_\beta |\dset|}\,\rd t.
\end{align}
Using the $1$-periodicity of $P_{\betaD}$, we compute for $\re(w)>0$,
\begin{align*}
\int_{0}^{\beta^{-c}}e^{P_{\betaD}\left(\log_\beta t\right)}t^{w-1}\,\rd t&=\sum_{k> c}\int_{\beta^{-k}}^{\beta^{1-k}}e^{P_{\betaD}\left(\log_\beta t\right)}t^{w-1}\,\rd t\\
&=\sum_{k> c}\beta^{-kw}\int_{1}^{\beta}e^{P_{\betaD}\left(\log_\beta t\right)}t^{w-1}\,\rd t\\
&=\frac{\beta^{-cw}\log\beta}{\beta^w-1}\int_{0}^{1}e^{P_{\betaD}\left(u\right)}\beta^{w u}\,\rd u.
\end{align*}
Substituting this identity with $w=s+\ell-\log_\beta|\dset|$ into \eqref{eqppd} and combining with \eqref{eq000} gives \eqref{eqzebme}.

To see that \eqref{eqzebme} provides a meromorphic continuation to $\bc$, let $\sigma=\re(s)$ and choose $N\ge 1-\sigma+\log_\beta |\dset|$. Then for $\ell\ge N$,
\begin{align*}
\left|\frac{\beta^{-(s+\ell-\log_\beta |\dset|)c}}{\beta^{s+\ell-\log_\beta |\dset|}-1}c_{\betaD}(\ell)
\int_{0}^{1}e^{P_{\betaD}\left(u\right)}\beta^{(s+\ell-\log_\beta |\dset|) u}\,\rd u\right|
&\ll |c_{\betaD}(\ell)|\beta^{-\rho_{\betaD}\ell}.
\end{align*}
Since $\sum_{\ell\ge 0}|c_{\betaD}(\ell)|\beta^{-\rho_{\betaD}\ell}<+\infty$ by definition of $\rho_{\betaD}$, the series in \eqref{eqzebme} converges absolutely and uniformly on compact subsets of $\bc$ avoiding the poles $p_{\betaD}(j,k)$. Computing the residues at $s=p_{\betaD}(j,k)$ and evaluating $\zeta_{\betaD}(-n)$ via $\Gamma(s+n+1)=\Gamma(s)\prod_{0\le j\le n}(s+j)$ follow by standard residue calculus.
\end{proof}

\begin{theorem}\label{th35}
If $b_k=\beta^k+O(\beta^{(1-\gamma)k})$ for some fixed $\gamma\in(0,1]$, then $\zeta_{\bd}(s)$ admits a meromorphic continuation to the half-plane $\re(s)>\log_\beta |\dset|-\gamma$, given by
\begin{align*}
\zeta_{\bd}(s)=-\frac{1}{\Gamma(s+1)}+\frac{1}{\Gamma(s)}\frac{\log\beta}{\beta^{s-\log_\beta |\dset|}-1}
\int_{0}^{1}e^{P_{\betaD}\left(u\right)}\beta^{(s-\log_\beta |\dset|) u}\,\rd u+\frac{h_{\bd}(s)}{\Gamma(s)},
\end{align*}
where $h_{\bd}(s)$ is holomorphic on $\re(s)>\log_\beta |\dset|-\gamma$.
\end{theorem}

\begin{proof}
Splitting the Mellin integral at $t=1$ as in the proof of Theorem \ref{th34} gives
\begin{align*}
\zeta_{\bd}(s)=\frac{1}{\Gamma(s)}\sum_{\lambda>0}\frac{r_{\bd}(\lambda)}{\lambda^s}\Gamma(s, \lambda)-\frac{1}{s\Gamma(s)}+\frac{1}{\Gamma(s)}\int_{0}^{1}Z_{\bd}\left(e^{-t}\right)t^{s-1}\,\rd t.
\end{align*}
Applying Lemma \ref{lem34} with $c=1$, we have for $\re(s)>\log_\beta |\dset|$,
\begin{align*}
\int_{0}^{1}Z_{\bd}\left(e^{-t}\right)t^{s-1}\,\rd t&=\int_{0}^{1}\left(e^{P_{\betaD}\left(\log_\beta t\right)}t^{s-1-\log_\beta |\dset|}+t^{s-1-\log_\beta |\dset|+\gamma}\left(\log\frac{2}{t}\right)^{\ind{\gamma=1}}B_{\bd}(t)\right)\,\rd t\\
&=\frac{\log\beta}{\beta^{s-\log_\beta |\dset|}-1}\int_{0}^{1}e^{P_{\betaD}\left(u\right)}\beta^{(s-\log_\beta |\dset|)u}\,\rd u\\
&\quad+\int_{0}^{1}B_{\bd}(t)\left(\log\frac{2}{t}\right)^{\ind{\gamma=1}}t^{s-1-\log_\beta |\dset|+\gamma}\,\rd t.
\end{align*}
Since $B_{\bd}(t)$ is bounded on $(0,1]$, the last integral converges absolutely and locally uniformly on the half-plane $\re(s)>\log_\beta |\dset|-\gamma$, defining a holomorphic function there. This completes the proof.
\end{proof}


\begin{thebibliography}{1}

\bibitem{MR4235264}
S.~Chow and T.~Slattery.
\newblock On {F}ibonacci partitions.
\newblock {\em J. Number Theory}, 225:310--326, 2021.

\bibitem{MR1690457}
J.~M. Dumont, N.~Sidorov, and A.~Thomas.
\newblock Number of representations related to a linear recurrent basis.
\newblock {\em Acta Arith.}, 88(4):371--396, 1999.

\bibitem{MR3237075}
D.-J. Feng, P.~Liardet, and A.~Thomas.
\newblock Partition functions in numeration systems with bounded multiplicity.
\newblock {\em Unif. Distrib. Theory}, 9(1):43--77, 2014.

\bibitem{MR2445243}
G.~H. Hardy and E.~M. Wright.
\newblock {\em An introduction to the theory of numbers}.
\newblock Oxford University Press, Oxford, sixth edition, 2008.
\newblock Revised by D.~R. Heath-Brown and J.~H. Silverman, with a foreword by
  Andrew Wiles.

\bibitem{MR2039503}
Y.~Katznelson.
\newblock {\em An introduction to harmonic analysis}.
\newblock Cambridge Mathematical Library. Cambridge University Press,
  Cambridge, third edition, 2004.

\bibitem{MR4215805}
M.~B. Nathanson.
\newblock Dirichlet series of integers with missing digits.
\newblock {\em J. Number Theory}, 222:30--37, 2021.

\bibitem{MR1374325}
I.~A. Pushkarev.
\newblock The ideal lattices of multizigzags and the enumeration of {F}ibonacci
  partitions.
\newblock {\em Zap. Nauchn. Sem. S.-Peterburg. Otdel. Mat. Inst. Steklov.
  (POMI)}, 223(Teor. Predstav. Din. Sistemy, Kombin. i Algoritm. Metody.
  I):280--312, 342, 1995.

\bibitem{MR1013117}
H.~L. Royden.
\newblock {\em Real analysis}.
\newblock Macmillan Publishing Company, New York, third edition, 1988.

\bibitem{MR1963498}
A.~Zygmund.
\newblock {\em Trigonometric series. {V}ol. {I}, {II}}.
\newblock Cambridge Mathematical Library. Cambridge University Press,
  Cambridge, third edition, 2002.
\newblock With a foreword by Robert A.~Fefferman.

\end{thebibliography}
\end{document}